\documentclass[10pt,a4paper]{amsart}
\usepackage{geometry}

\usepackage{etoolbox}
\apptocmd{\sloppy}{\hbadness 10000\relax}{}{}

\newcommand{\comment}[1]{} 

\usepackage{microtype}
\usepackage{babel}
\usepackage[utf8]{inputenc}
\usepackage{enumitem}
\usepackage{amsmath,amsthm,amssymb,amsfonts}
\usepackage{tikz-cd}
\usetikzlibrary{arrows}
\input xy
\xyoption{all}
\newdir{ >}{{}*!/-5pt/@{>}}

\newcommand{\Z}{\mathbb{Z}}
\newcommand{\N}{\mathbb{N}}

\newcommand{\cA}{\mathcal A}

\newcommand{\cE}{\mathcal E}
\newcommand{\cF}{\mathcal F}

\newcommand{\cK}{\mathcal K}

\newcommand{\cP}{\mathcal P}

\newcommand{\cR}{\mathcal R}

\newcommand{\cT}{\mathcal T}
\newcommand{\cU}{\mathcal U}

\newcommand{\fA}{\mathfrak{A}}

\newcommand{\fS}{\mathfrak{S}}

\newcommand{\fU}{\mathfrak{U}}

\newcommand{\fX}{\mathfrak{X}}

\newcommand{\topdf}{\texorpdfstring}

\newcommand{\BF}{\mathfrak{BF}}

\newcommand{\xto}{\xrightarrow}
\newcommand{\iso}{\overset{\sim}{\longrightarrow}}
\DeclareMathOperator{\ad}{ad}
\DeclareMathOperator{\reg}{reg}
\DeclareMathOperator{\cl}{cl}
\DeclareMathOperator{\coker}{coker}
\DeclareMathOperator{\colim}{colim}
\DeclareMathOperator{\id}{id}
\DeclareMathOperator{\im}{im}
\DeclareMathOperator{\supp}{supp}
\DeclareMathOperator{\inc}{inc}

\DeclareMathOperator{\End}{End}
\DeclareMathOperator{\can}{can}
\DeclareMathOperator{\ev}{ev}
\DeclareMathOperator{\inv}{inv}
\DeclareMathOperator{\res}{res}
\DeclareMathOperator{\ind}{ind}
\DeclareMathOperator{\sink}{sink}
\DeclareMathOperator{\forg}{forg}
\DeclareMathOperator{\tor}{Tor}
\def\gr{\operatorname{gr}}
\def\op{\operatorname{op}}

\newcommand{\cat}[1]{\mathsf{#1}}
\newcommand{\Alg}{\cat{Alg}_\ell}
\newcommand{\aha}{\Alg}
\newcommand{\ahas}{\Alg^\ast}
\newcommand{\hltimes}{{ \ \widehat{\ltimes}\ }}
\newcommand{\gBF}{\BF_{\gr}}
\newcommand{\elmat}{\varepsilon}
\newcommand{\ucov}[1]{\widetilde{#1}}
\newcommand{\GAlg}{{G-\Alg^\ast}}
\newcommand{\GgrAlg}{{G_{\gr}-\Alg^\ast}}
\newcommand{\AnyAlg}{\fA}
\newcommand{\lra}{\longrightarrow}
\newcommand{\lspan}{\mathsf{span}_\ell}
\newcommand{\eps}{\epsilon}
\newcommand{\triqui}{\triangleleft}
\usepackage{xcolor}

\definecolor{thmcol}{RGB}{0, 81, 255}
\definecolor{citecol}{RGB}{0, 81, 255}
\definecolor{linkcol}{RGB}{0, 81, 255}
\definecolor{urlcol}{RGB}{0, 81, 255}

\numberwithin{equation}{section}

\theoremstyle{plain}
\newtheorem*{thm*}{Theorem}
\newtheorem{thm}[equation]{Theorem}
\newtheorem{lem}[equation]{Lemma}
\newtheorem{coro}[equation]{Corollary}
\newtheorem{prop}[equation]{Proposition}
\newtheorem{conv}[equation]{Convention}

\theoremstyle{definition}
\newtheorem{defn}[equation]{Definition}

\newtheorem{ex}[equation]{Example}

\newtheorem{stan}[equation]{Standing assumption}

\theoremstyle{remark}
\newtheorem{rmk}[equation]{Remark}
\newtheorem{rem}[equation]{Remark}
\newtheorem{nota}[equation]{Notation}
\newtheorem*{ack}{Acknowledgements}
\usepackage[colorlinks=true, pagebackref=true]{hyperref}
\hypersetup{
    colorlinks,
    citecolor=citecol,
    linkcolor=linkcol,
    urlcolor=urlcol
}
\usepackage{amsrefs}
\usepackage{url}

\title[Graded \topdf{$K$}{K}-theory and Leavitt path algebras]{Graded  $K$-theory and Leavitt
path algebras}
\author{Guido Arnone and Guillermo Corti\~nas}
\address{Dep. Matem\'atica-IMAS\\ Facultad de Ciencias Exactas y Naturales\\
Universidad de Buenos Aires\\ Argentina}
\thanks{Both authors were partially supported by grant UBACyT 256BA from Universidad de Buenos Aires. The first named author was supported by a PhD fellowship from Universidad de Buenos Aires. The second named author was supported by CONICET and partially supported by grants PICT 2017-1395 from Agencia Nacional de Promoci\'on Cient\'\i fica y T\'ecnica, and PGC2018-096446-B-C21 from the Spanish Ministerio de Ciencia e Innovaci\'on.}

\date{}

\begin{document}

\begin{abstract} Let $G$ be a group and $\ell$ a commutative unital $\ast$-ring with 
an element $\lambda \in \ell$ such that $\lambda + \lambda^\ast = 1$.
We introduce variants of hermitian bivariant $K$-theory for $\ast$-algebras
equipped with a $G$-action or a $G$-grading. For any graph $E$ with finitely many vertices
and any weight function $\omega \colon E^1 \to G$, a distinguished triangle 
for $L(E)=L_\ell(E)$ in the hermitian $G$-graded bivariant $K$-theory category $kk^h_{G_{\gr}}$ is obtained, describing $L(E)$ as a cone of a matrix with coefficients in $\Z[G]$ associated to the incidence matrix of $E$ and the weight $\omega$.
In the particular case of the standard $\Z$-grading, and under mild assumptions on $\ell$, 
we show that the isomorphism class of $L(E)$ in $kk^h_{\Z_{\gr}}$ is determined by the graded
Bowen-Franks module of $E$. We also obtain results for the graded and hermitian graded $K$-theory of $\ast$-algebras in general and Leavitt path algebras in particular which are of independent interest, including hermitian and bivariant versions of Dade's theorem and of Van den Bergh's exact sequence relating graded and ungraded $K$-theory. 
\end{abstract}
\maketitle

\section{Introduction}\label{sec:intro}

A (directed) graph $E$ consists of sets $E^0$ of vertices and $E^1$ of
edges and source and range maps $s,r:E^1\to E^0$. A vertex $v\in E^0$ is a \emph{sink} or an \emph{infinite emitter} the number $\#s^{-1}\{v\}$ of edges it emits is zero or infinite, and is \emph{regular} otherwise. We say that $E$ is \emph{row-finite} if it has no infinite emitters and \emph{finite} if $E^0$ and $E^1$ are finite. We write $\reg(E)\subset E^0$ for the
subset of regular vertices. Let $A_E\in\Z^{\reg(E)\times E^0}$ be the matrix whose
$(v,w)$ entry is the number of edges from $v$ to $w$; put $I\in\Z^{E^0\times\reg(E)}$, 
$I_{v,w}=\delta_{v,w}$. Let $C_\infty=\langle \sigma\rangle$ be the 
infinite cyclic group; write $\Z[\sigma]$ for its group ring. 
The \emph{graded Bowen-Franks} $\Z[\sigma]$-module of $E$ is 
\begin{equation}\label{eq:bfgr}
    \BF_{\gr}(E)=\coker(I-\sigma A_E^t).
\end{equation}
If $E^0$ is finite, we put $[1]_E=\sum_{v\in E^0}[v]\in\BF_{\gr}(E)$. Fix a commutative unital ring $\ell$ with involution and let $L(E)$ be the Leavitt path algebra of $E$ over $\ell$ \cite{lpabook}, equipped with its canonical $\Z$-grading and involution. Write
$K_*^{\gr}(L(E))$ for the $K$-theory of graded, finitely generated projective
$L(E)$-modules. The group $K_0^{\gr}(L(E))$ is canonically a preordered $\Z[\sigma]$-module.  If $E^0$ is row-finite, then
there is a natural ordered $\Z[\sigma]$-module isomorphism (Corollary \ref{coro:bf-k0=k0gr})
\[
\BF_{\gr}(E)\otimes K_*(\ell)\iso K_*^{\gr}(L(E)),
\]
which, when $E^0$ is finite, maps
$[1]_E\otimes [1_\ell]\mapsto [1_{L(E)}]$.
Let $E$ and $F$ be finite graphs. The graded classification conjecture
\cite{hazrat}*{Conjecture 1} for finite graphs says that, when $\ell$ is a field, 
the existence of an order preserving
$\Z[\sigma]$-module isomorphism $\BF_{\gr}(E)\iso \BF_{\gr}(F)$ 
mapping $[1]_E\mapsto [1]_F$ implies that $L(E)\cong L(F)$ 
as $\Z$-graded algebras. In this paper we begin the study of this conjecture
in terms of the bivariant algebraic $K$-theory of algebras graded over a group $G$
$j_{G_{\gr}}:G_{\gr}-\Alg\to kk_{G_{\gr}}$ of \cite{kkg}, and of a hermitian
version $j^h_{G_{\gr}}:G_{\gr}-\Alg^\ast\to kk^h_{G_{\gr}}$ that we introduce here. As in the ungraded case \cite{kkh}, to define $kk^h_{G_{\gr}}$ we assume that there exists $\lambda\in\ell$ such that 
\begin{equation}\label{eq:lambda}
\lambda + \lambda^\ast = 1.    
\end{equation}
In Theorem \ref{lem:liftkkz} we prove the following.
\begin{thm}\label{intro:liftkkz}
Let $\ell$ be a ring with involution satisfying \eqref{eq:lambda}, and let $E$ and $F$ be graphs
with finitely many vertices. Write $L(E)$ and $L(F)$ for their Leavitt path $\ell$-algebras. Any $\Z[\sigma]$-module isomorphism $\xi:\BF_{\gr}(E)\iso \BF_{\gr}(F)$ lifts to an isomorphism 
$j^h_{\Z_{\gr}}(L(E))\iso j^h_{\Z_{\gr}}(L(F))$ in $kk^h_{\Z_{\gr}}$.
\end{thm}

The exact meaning of lifting in the theorem above is made precise in 
Theorem \ref{lem:liftkkz}. The forgetful functor $\Alg^\ast\to \Alg$ induces
a functor $kk^h_{\Z_{\gr}}\to kk_{\Z_{\gr}}$, so under the hypothesis of 
Theorem \ref{intro:liftkkz} we also have an isomorphism
$j_{\Z_{\gr}}(L(E))\cong j_{\Z_{\gr}}(L(F))$. Further, we show that under mild
additional assumptions on the homotopy algebraic $K$-theory $KH_*(\ell)$ of $\ell$, which are
satisfied, for example, if $\ell$ is a field, a PID, or a  noetherian
regular local ring, the Bowen-Franks module classifies Leavitt path algebras
in both $kk_{\Z_{\gr}}$ and $kk_{\Z_{\gr}}^h$. In Theorem \ref{thm:classif} we prove
the following.
 
\begin{thm} \label{intro:classif}
Assume that $\ell$ satisfies \eqref{eq:lambda}, that $KH_{-1}(\ell) = 0$ and that the canonical morphism $\Z \to KH_0(\ell)$ 
is an isomorphism. Then for each pair of graphs $E$ and $F$ with finitely many vertices, the following 
are equivalent:
\begin{itemize}
    \item[(i)] The algebras $L(E)$ and $L(F)$ are $kk^h_{\Z_{\gr}}$-isomorphic.
    \item[(ii)] The algebras $L(E)$ and $L(F)$ are $kk_{\Z_{\gr}}$-isomorphic.
    \item[(iii)] The $\Z[\sigma]$-modules $\gBF(E)$ and $\gBF(F)$ are isomorphic.  
\end{itemize}
\end{thm}
We also obtain several other results about (hermitian) graded $K$-theory of  ($\ast$-) algebras in general and Leavitt path algebras in particular which we think are of independent interest. 
Building upon work of Ara, Hazrat, Li and Sims in \cite{steinberg} and Preusser in \cite{preuhyper}, we show in Theorems \ref{thm:k-cross=k-gr} and \ref{thm:kh-cross=kh-gr} that if $R$ is a ($*$-) ring, graded over a group $G$, and having (self-adjoint) graded local units, then the (hermitian) graded $K$-theory of $R$ is the (hermitian) $K$-theory of the crossed product 
\begin{equation}\label{intro:kcross=kgr}
K_*^{G_{\gr}}(R)=K_*(G\hltimes R),\,\, 
K_*^{h,G_{\gr}}(R)=K^h_*(G\hltimes R).
\end{equation}
The definition of $G\hltimes R$ is recalled in Subsection \ref{subsec:cropro}; the same formulas hold for homotopy algebraic and homotopy hermitian $K$-theory under no unitality assumptions (see Corollary \ref{coro:kkhgr-ell=khcross}).  For example if $E$ is row-finite and $R=L(E)$ with its canonical $\Z$-grading, then $\Z\hltimes L(E)=L(\ucov{E})$ is the Leavitt path algebra of the universal covering of $E$ (Proposition \ref{prop:crosscover}). We use this to compute, for a row-finte graph $E$ and an algebra $R$ with local units, equipped with the trivial grading, the graded $K$-theory of the algebra $L(E)\otimes R$  and, in case $R$ is a $\ast$-algebra with self-adjoint graded local units, also its hermitian graded $K$-theory; we show in Corollary \ref{coro:bf-k0=k0gr} that
\begin{equation}\label{intro:kgrle}
K_*^{\gr}(L(E)\otimes R)=\gBF(E)\otimes K_*(R),\,\, K_*^{h,\gr}(L(E)\otimes R)=\gBF(E)\otimes K^h_*(R).
\end{equation}
Again the same formulas hold for homotopy algebraic and homotopy hermitian $K$-theory without any unitality assumptions on $R$. We also consider gradings on $L(E)$ with values on a group $G$ for an arbitrary graph $E$. These are associated to functions $\omega: E^1\to G$; we write $L_\omega(E)$ for $L(E)$ with the induced grading. There is a universal covering $\widetilde{E}=(\widetilde{E,\omega})$ and again $G\hltimes L(E)=L(\widetilde{E,\omega})$. 
Let 
\begin{equation}\label{intro:weightind}
A_{\omega}\in\Z[G]^{\reg(E)\times E^0},\,\, (A_{\omega})_{v,w}=\sum_{v\overset{e}\to w}\omega(e). 
\end{equation}
Observe that for $G=\Z$ and $\omega$ the constant grading above, $A_{\omega}=\sigma A_{E}\in\Z[\sigma]^{\reg(E)\times E^0}$. Define $\gBF(E,\omega) := \coker(I-A_\omega^t)$.
We show in Corollary \ref{coro:kle} that if $R$ is a trivially $G$-graded, unital, regular supercoherent algebra, then there is an exact sequence
\begin{equation}\label{intro:kle}
0\to\gBF(E,\omega)\otimes K_{n}(R)\to  K_n^{G_{\gr}}(L_\omega(E)\otimes R)\to \ker((I-A_\omega^t)\otimes K_{n-1}(R))\to 0.
\end{equation}
If in addition $R$ is a $*$-algebra and $2$ is invertible in $R$ we have a similar sequence for $G$-graded hermitian $K$-theory. Both \eqref{intro:kle} and its hermitian counterpart hold for homotopy $K$-theory for every trivially $G$-graded ($\ast$-) algebra $R$, where in the hermitian case we assume that the ground ring $\ell$ satisfies \eqref{eq:lambda}. We show in Proposition \ref{prop:end-el} that there is a ring isomorphism
$kk^h_{G_{\gr}}(\ell,\ell) = kk^h(\ell,\ell)\otimes \Z[G^{\op}]$. In particular, if $E^0$ is finite, the matrix $I-A_{\omega}^t$ defines an element of $kk^h_{G_{\gr}}(\ell^{\reg(E)},\ell^{E^0})$, and Corollary \ref{coro:g-triang} says that there is a distinguished triangle in $kk^h_{G_{\gr}}$ which, omitting $j_{G_{\gr}}^h$, has the form
\begin{equation}\label{intro:triang}
\xymatrix{\ell^{\reg(E)}\ar[r]^(.6){I-A_{\omega}^t}&\ell^{E^0}\ar[r]& L_\omega(E).}
\end{equation}

Recall that a $G$-graded ring $A$ is \emph{strongly graded} if $A_gA_h=A_{gh}$ for all $g,h\in A$. A theorem of Dade characterizes strongly $G$-graded unital rings in terms of module categories; we obtain a hermitian variant
of Dade's theorem for $\ast$-rings in terms of module categories with duality (Theorem \ref{thm:h-dade}). We also show (Theorem \ref{thm:dade}) that if $B$ is a strongly $G$-graded $\ast$-algebra, then the canonical inclusion 
\begin{equation}\label{intro:bivadade}
B_1\subset G\hltimes B    
\end{equation} 
is a $kk^h$-equivalence. For example if $A$ is any $\Z$-graded $\ast$-algebra, then $A[t,t^{-1}]$ is strongly $\Z$-graded. Using this together with 
\eqref{intro:triang} we show in Theorem \ref{thm:vdb} that there is a distinguished triangle in $kk^h$
\begin{equation}\label{intro:vdb}
\xymatrix{
\Z\hltimes A\ar[rr]^{1-\Z\hltimes\sigma}&& \Z\hltimes A\ar[r] & A.
}
\end{equation}
As a consequence of \eqref{intro:vdb} and of the homotopy hermitian $K$-theory version of \eqref{intro:kcross=kgr}, we obtain an exact sequence
\[
KH^h_{n+1}(A)\to KH_n^{h,\Z_{\gr}}(A)\xto{1-\sigma} KH_n^{h,\Z_{\gr}}(A)\to KH^h_{n}(A)
\]
for any $n\in \Z$ and any $\ast$-algebra $A$, and a similar sequence for non-hermitian homotopy $K$-theory of any ring $A$. For sufficiently regular $A$, $K^h$ and $K$ may be substituted for $KH^h$ above; in particular we recover the classical Van den Bergh exact sequence of \cite{vdb} for regular noetherian $A$, as well as a hermitian variant of his sequence (Corollary \ref{coro:vdb}).

The rest of this article is organized as follows. In Section \ref{sec:gbasic}
we recall basic definitions, notations and properties for algebras equipped
with an action of, or a grading over a group $G$. Section \ref{sec:cross}
contains the proof of the identities \eqref{intro:kcross=kgr} in Theorems
\ref{thm:k-cross=k-gr} and \ref{thm:kh-cross=kh-gr}; the basic idea is to use
the category isomorphism between graded $R$-modules and $G\hltimes R$-modules
due to Ara, Hazrat, Li and Sims \cite{steinberg} and the fact that the latter
preserves finite generation, proved in Preusser's article \cite{preuhyper}, and to check that the category
isomorphism intertwines the relevant duality functors. As an application we
also establish Theorem \ref{thm:h-dade}, which is a hermitian variant of
Dade's theorem. Section \ref{sec:kleav} is concerned with the proof of
\eqref{intro:kgrle}, obtained in Corollary \ref{coro:bf-k0=k0gr} as a
consequence of the more general Theorem \ref{thm:bf-k0=k0gr} which establishes
a similar formula with $K_n^h$ and $K_n$ replaced by any functor $H$ from
$\Z$-graded $\ast$-rings to abelian groups satisfying some mild assumptions. We also prove in Theorem \ref{thm:reflinj} that if $\ell$ is a field, then $K_0^{\gr}$ reflects injectivity of $\Z$-graded algebra homomorphisms
$L(E)\to R$ with $E$ finite. Thus if $\phi$ is such a homomorphism and $K_0^{\gr}(\phi)$ is injective, then so must be $\phi$. 
Section \ref{sec:stab} concerns appropriate notions of stability for functors
defined on the categories $G-\Alg^*$ of $G$-$\ast$-algebras and
$G_{\gr}-\Alg^\ast$ of $G$-graded $\ast$-algebras. Section \ref{sec:kkhgr}
introduces triangulated categories $kk^h_G$ and $kk^h_{G_{\gr}}$ and functors 
\[
j^h_G \colon \GAlg \to kk^h_G, \qquad
j^h_{G_{\gr}} \colon \GgrAlg \to kk_{G_{\gr}}^h.
\]
Each of these functors satisfies a universal property, which essentially says
that it is universal among excisive, homotopy invariant and stable homology
theories. In Section \ref{sec:adj} we show that the adjointness theorems
proved by Ellis in \cite{kkg} remain valid in the hermitian setting. In
particular for a subgroup $H\subset G$, the induction and restriction functors
define an adjoint pair
$kk^h_{H}\leftrightarrows kk^h_{G}$ (Theorem \ref{thm:ind-res}) and the cross
product functors induce inverse category equivalences
$kk^h_G\overset{\sim}\longleftrightarrow kk^h_{G_{\gr}}$ (Theorem \ref{thm:baaj-skandalis}). 
Let $KH^h$ be the homotopy hermitian
$K$-theory of \cite{kkh}*{Section 3}. In Section \ref{sec:enrich} we compute the coefficient ring $kk^h_{G_{\gr}}(\ell,\ell)$. We show in Proposition \ref{prop:end-el} that composing the canonical isomorphisms between $kk$-groups provided by 
Theorems \ref{thm:ind-res} and \ref{thm:baaj-skandalis} one obtains
a ring isomorphism
\begin{equation}\label{intro:coeffring}
\Z[G^{\op}]\otimes KH^h_0(\ell)\iso kk_{G_{\gr}}^h(\ell,\ell).    
\end{equation}
This leads to a left action of $G$ by natural transformations on $kk^{h}_{G_{\gr}}(\ell,A)$ for all $A\in G_{\gr}-\ahas$; Lemma \ref{lem:cdot=cdot'} shows this action agrees with that induced by the $G$ action on $G\hltimes A$ via the isomorphism $kk^{h}_{G_{\gr}}(\ell,A)=kk^h(\ell, G\hltimes A)$. When $G$ is abelian, the tensor product of graded algebras is again $G$-graded; this together with the isomorphism \eqref{intro:coeffring} leads to an enrichment of $kk^h_{G_{\gr}}$ over $\cat{mod}_{\Z[G]}$, which is described by Proposition \ref{prop:multiabel}. 
Section \ref{sec:dade} is devoted to the proof of Theorem \ref{thm:dade}, 
which says that if $B$ is strongly $G$-graded, then the map 
\eqref{intro:bivadade} is a $kk^h$-equivalence.

Let $L(E)\cong C(E)/\cK(E)$ be the usual presentation
 of the Leavitt path algebra as a quotient of the Cohn algebra
 \cite{lpabook}*{Proposition 1.5.5}. Let $\omega:E^1\to G$ be a weight 
 function
 and let $L_\omega(E)$, $C_\omega(E)$ and $\cK_\omega(E)$ be the same algebras 
 equipped with the associated $G$-gradings. In Section
\ref{sec:triang} we study the Cohn extension
\begin{equation}\label{ext:cohn}
0\to \cK_\omega(E)\to C_\omega(E)\to L_\omega(E)\to 0
\end{equation}
in $kk^h_{G_{\gr}}$. Propositions \ref{prop:q-iso} and \ref{prop:phi-iso} show that for any excisive, homotopy invariant and stable homology theory $H$ of $G$-graded $\ast$-algebras which commutes with sums of sufficiently high cardinality we have  $H(\cK_\omega(E))\cong H(\ell)^{(\reg(E))}$ and $H(C_\omega(E))\cong H(\ell)^{(E^0)}$. The map $\xi:H(\ell)^{(\reg(E))}\to H(\ell)^{(E^0)}$ induced by the inclusion $\cK_\omega(E)\subset C_\omega(E)$ is computed in Theorem \ref{thm:triang}, and so we obtain a triangle
\begin{equation*}
\xymatrix{H(\ell)^{(\reg(E))}\ar[r]^(.6){I-A_{\omega}^t}&H(\ell)^{(E^0)}\ar[r]& H(L_\omega(E)).}
\end{equation*}
When $E^0$ is finite, we may take $H=j_{G_{\gr}}^h$; this gives  triangle \eqref{intro:triang} (Corollary \ref{coro:g-triang}).

In Section \ref{sec:vdb} we apply Theorems \ref{thm:dade} and \ref{thm:triang} to prove Theorem \ref{thm:vdb} which establishes the distinguished triangle \eqref{intro:vdb}, and Corollary \ref{coro:vdb} which derives long exact sequences relating graded and ungraded (hermitian, homotopy) $K$-theory.

In Section \ref{sec:classif} we use triangle \eqref{intro:triang} to prove Theorems \ref{intro:liftkkz} and \ref{intro:classif} as Theorems \ref{lem:liftkkz} and \ref{thm:classif}.

\begin{ack} A good part of the material of Sections \ref{sec:stab}, \ref{sec:kkhgr}, \ref{sec:adj}, \ref{sec:enrich}, \ref{sec:triang} and \ref{sec:classif} first appeared in the diploma thesis of the first author \cite{tesigui}. 
\end{ack}

\section{Preliminaries on algebras, involutions, gradings, and actions}\label{sec:gbasic}

\subsection{Algebras and involutions}
A commutative, unital ring $\ell$ with involution $\ast$
will be fixed throughout the article. By an \emph{$\ell$-algebra}
we mean a symmetric $\ell$-bimodule $A$ together with 
an associative multiplication $A \otimes_\ell A \to A$. An involution
of an $\ell$-algebra is an additive morphism $\ast \colon A \to A$
such that 
\[
(a^\ast)^\ast = a, \quad 
(ab)^\ast  = b^\ast a^\ast, \quad 
(\mu a)^\ast = \mu^\ast a^\ast
\]
for all $a,b \in A$ and $\mu \in \ell$. A \emph{$\ast$-algebra} is an $\ell$-algebra
together with an involution; a \emph{$\ast$-morphism} $f \colon A \to B$ between 
$\ast$-algebras is an $\ell$-algebra morphism such that $f(a^\ast) = f(a)^\ast$ for each
$a \in A$. A two-sided ideal 
$I \triangleleft A$ is a \emph{$\ast$-ideal} if in addition it is an $\ell$-submodule of $A$ 
such that $I^\ast \subset I$. The category of $\ast$-algebras together with $\ast$-morphisms will be denoted 
$\Alg^\ast$. 
Tensor products of algebras are taken over $\ell$; we write $\otimes$ for $\otimes_\ell$. We also use $\otimes$ for tensor products of abelian groups, e.g. $K_0(R)\otimes K_0(S)=K_0(R)\otimes_{\Z} K_0(S)$. We write $R\otimes_\Z S$ for the tensor product of rings which are not necessarily $\ell$-algebras. 

An element $u$ in a unital $\ast$-algebra $R$
is said to be \emph{unitary} if $uu^\ast = u^\ast u = 1$. If $\eps \in R$ is both
central and unitary, we say that $\phi \in R$ is
\emph{$\eps$-hermitian} if $\phi=\eps\phi^*$.
An $\eps$-hermitian unit $\phi \in R^\times$ gives rise to an involution
\[
    (-)^\phi \colon R \to R, \quad x \mapsto \phi^{-1}x^\ast \phi.
\]
We write $R^\phi$ for $R$ viewed as a $\ast$-algebra with this involution.
Similarly, for a $*$-ideal $A\triqui R$ the notation $A^\phi$ will be employed for $A$ equipped with the involution induced from $R^\phi$.
\numberwithin{equation}{subsection}
\begin{nota}
\label{nota:pr-mor} An element $a$ in $\ast$-algebra $A$ is \emph{self-adjoint} if $a=a^*$. A \emph{projection} 
$p\in A$ is a self-adjoint idempotent element. Such an element determines 
a $\ast$-morphism which we will also call $p$.
\end{nota}

\subsection{Actions and gradings}

From now on, we fix a group $G$, which we regard as a one object category. The category $\GAlg$ of \emph{$G$-$\ast$-algebras} is the category of functors $G\to\Alg^\ast$.
A \emph{$G$-graded $\ast$-algebra} is a $G$-graded $\ell$-algebra $A=\bigoplus_{g\in G}A_g$ 
equipped with an involution such that $A_g^\ast \subset A_{g^{-1}}$ for all $g \in G$. 
The degree of a homogeneous element $a \in A$ will be denoted $|a|$. We write 
$\GgrAlg$ for the category of $G$-graded $\ast$-algebras together with $\ast$-homomorphisms
that are homogeneous of degree $1 \in G$. The letter 
$\AnyAlg$ will refer to either $\GAlg$ or $\GgrAlg$.

\begin{rmk} \label{rmk:tensor-gr} If $A$ is a $\ast$-algebra and $B$ a $G$-graded $\ast$-algebra, 
their tensor product $A \otimes_\ell B$ is $G$-graded by defining $|a \otimes b| = |b|$ for each 
$a \in A, b \in B$ with $b$ homogeneous. If both $A$ and $B$ are $G$-graded $\ast$-algebras and $G$ is abelian, then $A \otimes_\ell B$ can be made into a $G$-graded $\ast$-algebra by 
defining $|a \otimes b| = |a||b|$ for each pair of homogeneous elements $a \in A, b \in B$. When $G$ is not abelian, 
this assignment fails to be compatible with multiplication.
\end{rmk}

\subsection{Matrix algebras}
Let $X$ be a set and $A$ a $\ast$-algebra. Recall that Wagoner's cone is the ring
\begin{equation}\label{eq:wacon}
C_X A := \{f \colon X \times X \to A : |\supp f(x,-)|, |\supp f(-,x)| < \infty \ (\forall x \in X)\}
\end{equation}
with the operations given by the pointwise sum and the convolution product. Its canonical 
involution is defined to be $f^\ast(x,y) = f(y,x)^\ast$ for each $f \in C_X A$. It 
contains Karoubi's cone 
\begin{equation}\label{eq:karcon}
\Gamma_X A = \{f \colon X \times X \to A : |\im f| < 
\infty \text{ and } (\exists N \geq 1) \text{ s.t. } |\supp(x,-)|, |\supp(-,x)| 
\leq N (\forall x \in X)\}
\end{equation}
as a $*$-subalgebra. We will also consider the $\ast$-ideal of $\Gamma_X A$ consiting of 
$X$-indexed finitely supported matrices
\[
M_X A := \{f \colon X \times X \to A : \supp(f) \text{ is finite}\}.
\]
In the  case $X=\N$ we use specific notation; we write $M_\infty A:=M_\N A$.
\begin{conv} When $A = \ell$, we omit it from the notation in all matrix algebras defined 
above.
\end{conv}
There is an isomorphism of $\ast$-algebras 
$M_X A \cong M_X \otimes_\ell A$; we write $\elmat_{x,y} \in M_X$ 
for the characteristic function of the pair $(x,y) \in X \times X$ and
\[
 \iota_x \colon A \to M_X A, \quad a \mapsto \elmat_{x,x} \otimes a
\]
for the inclusion in the diagonal entry corresponding to $x \in X$.

If $A$ is a $G$-$\ast$-algebra and $X$ is a $G$-set, then the algebras $C_X A$, $\Gamma_X A$
and $M_X A$ defined above
are $G$-$\ast$-algebras with action $(g \cdot f)(x,y) = g \cdot f(g^{-1}x, g^{-1}y)$. 
The analogue situation for $G$-graded $\ast$-algebras requires some preliminary definitions. 
A \emph{$G$-graded set} is a set $X$ together with a function 
$\omega\colon X \to G$, which we call a \emph{weight}. 
A morphism of $G$-graded sets $(X,\omega) \to (Y,\nu)$ is a map $f\colon X\to Y$
such that $\nu\circ f = \omega$. We shall often drop the weight from our vocabulary, and say simply that $X$ is a graded set; in this case we write $|\ \ |$ for the weight function of $X$.

\begin{defn} Let $B$ be a $G$-graded $\ast$-algebra and $X$ a $G$-graded set. Let $C_X^\circ B$ be the $\ell$-linear span of those elements $f\in C_X B$ such that $f(x,y)$ is homogeneous for each $(x,y) \in X\times X$ and such that the function $(x,y) \mapsto |x||f(x,y)||y|^{-1}$ is constant. The $*$-algebra $C_X^\circ B$ is $G$-graded;  its homogeneous component of degree $g \in G$ is
\[
(C_X^\circ B)_g := \{f \in C_X^\circ B : |x| \cdot |f(x,y)| = g|y| \ (\forall x,y \in X)\}.
\]
Set $\Gamma_X^\circ B := C_X^\circ B \cap \Gamma_X B$; one checks that $\Gamma_X^
\circ B\subset C_X^\circ B$ is $G$-graded $\ast$-subalgebra.
\end{defn}
\subsection{Matricial stability}\label{subsec:matstab}
Let $X$ be a set and $\cat{C}$ a  category; a functor $F \colon \AnyAlg\to \cat{C}$  
is called \emph{$M_X$-stable} if it sends each inclusion 
$\iota_x \colon A \to M_X A$ with $A \in \AnyAlg$ and $x \in X$ to an isomorphism. By the 
argument of \cite{kkh}*{Lemma 2.4.1}, this is equivalent to asking that the natural map $F(\iota_x)$ be an isomorphism for a fixed
$x \in X$.
\subsection{Crossed products}\label{subsec:cropro}
We now recall the definitions 
of crossed products for $G$-$\ast$-algebras and $G$-graded $\ast$-algebras. 
We adopt the notations of \cite{kkg}. Let $\ell[G]$ be the group algebra of $G$ 
and $A$ a $G$-$\ast$-algebra. The \emph{crossed product} $A
\rtimes G$ is the $\ell$-module $A \otimes_\ell \ell[G]$ together
with the multiplication defined by the rule
$(a \rtimes g) \cdot (b \rtimes h) = a(g \cdot b) \rtimes gh$.
Here $a\rtimes g$ is notation for the elementary tensor $a \otimes g$. The 
crossed product is a $G$-graded $\ast$-algebra with homogeneous components
$(A \rtimes G)_g = A \rtimes g = \lspan \{a \rtimes g : a \in A \}$    
and involution
$(a \rtimes g)^\ast = g^{-1} \cdot a^\ast \rtimes g^{-1}$.

If $B$ is a $G$-graded $\ast$-algebra, we may regard it as a comodule algebra over the (nonunital) Hopf algebra $\ell^{(G)}$. We write $G\hltimes B=\ell^{(G)}\ltimes B$ for the \emph{crossed product}. As an $\ell$-module, $G\hltimes B$ is just $\ell^{(G)}\otimes B$; we write $\chi_g\in\ell^{(G)}$ for the characteristic function and put $\chi_g\hltimes b:=\chi_g\otimes b$. As a $G$-$\ast$-algebra, $G\hltimes B$ embeds in $M_GB$ via the identification $\chi_g \hltimes b=\elmat_{g,g|b|} \otimes b$ for 
$g \in G$ and homogeneous $b\in B$. Thus for $g,h\in G$ and $b,c\in B$ with $b$ homogeneous, we have
\begin{gather*}
(\chi_g\hltimes b)(\chi_h\hltimes c)=\delta_{h,g|b|}\chi_g\hltimes bc,\\
(\chi_g\hltimes b)^\ast=\chi_{g|b|}\hltimes b^*,\\
h\cdot (\chi_g\hltimes c)=\chi_{hg}\hltimes c.
\end{gather*}

\begin{rem}\label{rem:smash=crossed}
 In \cite{steinberg}*{Definition 2.1} the smash product  $A\#G$ of a $G$-graded algebra $A$ is defined. One checks that, for the grading $A^{\op}_g := A_{g^{-1}}$ ($g \in G$),  the map 
\[
A\#G \to (G \hltimes A^{\op})^{\op} , \quad ap_g \mapsto \chi_{g^{-1}}\hltimes a
\]
is an $G$-algebra isomorphism. 
\end{rem}

Consider the functors
\begin{equation}\label{fun:crossed}
    - \rtimes G \colon \GAlg \longleftrightarrow \GgrAlg \colon G \hltimes -.
\end{equation}
\begin{prop}\label{prop:impri} Let $A$ be a $G$-$\ast$-algebra and $B$ a $G$-graded $\ast$-algebra. We have a natural isomorphism of $G$-$\ast$-algebras $G\hltimes (A\rtimes G)\cong M_GA$ and a natural isomorphism of $G$-graded $\ast$-algebras $(G\hltimes B)\rtimes G \cong M_GB$. 
\end{prop}
\begin{proof} One checks that the natural isomorphisms of \cite{kkg}*{Proposition 7.4} are $\ast$-homomorphisms. 
\end{proof}

\subsection{Covering graphs}\label{subsec:covers}
Let $G$ be a group. The \emph{covering} of a graph $E$ associated to a weight $\omega: E^1\to G$ is the graph $\ucov{E}=\ucov{(E,\omega)}$ with vertices and edges defined by
\[
    \ucov{E}^i = E \times G \qquad (i = 0,1).
\]
Write $v_g = (v,g)$ and $e_g = (e,g)$ for each $v \in E^0, e \in E^1$; the source and range functions of $\ucov{E}$ are defined by
\[
    s(e_g) = v_g,  \quad r(e_g) = r(e)_{g\omega(e)}.
\]
Observe that left multiplication on the $G$ component gives an action of $G$ on $\tilde{E}$ by graph automorphisms. In particular $L(\tilde{E})$ is a $G$-algebra. 

\begin{prop}\label{prop:crosscover}
Let $E$ be a graph, $\omega:E^1\to G$ a weight, $L_\omega (E)$ the Leavitt path algebra equipped with its canonical involution and the $G$-grading induced by $\omega$, and $\ucov{E}=\ucov{(E,\omega)}$ the associated covering graph. There is an isomorphism of $G$-$*$-algebras $L\ucov{E}\iso G\hltimes L_\omega(E)$.
\end{prop}
\begin{proof}
By  \cite{preuhyper}*{Example 7}, this is a particular case of \cite{preuhyper}*{Proposition 74}.
\end{proof}

\section{Graded \topdf{$K$}{K}-theory and crossed products}\label{sec:cross}
\numberwithin{equation}{subsection}
\subsection{Graded \topdf{$A$}{A} modules vs. \topdf{$G\hltimes A$}{GxA}-modules}\label{subsec:gramodglamod}
A $G$-graded ring $A$ has \emph{graded local units} if for every finite subset $\cF\subset A$ there exists a homogeneous idempotent $e$ such that $\cF\subset eAe$.

\begin{rmk}
Recall that the set $\cE$ of homogeneous idempotents of a $G$-graded ring $A$
can be equipped with a partial order; concretely, 
we set $e \leq f$ whenever $ef = fe=e$. If $A$ has graded local units, then for any cofinal subset $\cU\subset \cE$ we have $A = \colim_{\cE \ni e} eAe=\colim_{\cU\ni e}eAe$ as $G$-graded rings. We call any such $\cU$ a set of \emph{graded local units} of $A$. 
\end{rmk}

In \cite{steinberg}*{Proposition 2.5} Ara, Hazrat, Li and Sims show that if $A$ is a $G$-graded ring that 
has graded local units, then there is an isomorphism between the categories of unital left 
$G$-graded $A$-modules and  of unital left modules over 
the crossed product ring $A\#G$ of \cite{steinberg}*{Definition 2.1}, which by Remark \ref{rem:smash=crossed} is isomorphic to $(G \hltimes A^{\op})^{\op}$. 
Restating their result 
in terms of $G \hltimes A$, we obtain inverse isomorphisms of categories of unital right modules
\begin{equation}\label{mor:mod-cat-iso}
\Psi \colon  G_{\gr}-\cat{mod}_A \overset{\sim}{\longleftrightarrow} \cat{mod}_{G\hltimes A} \colon \Phi
\end{equation}
which we proceed to describe. 
The functor $\Psi$ sends a $G$-graded module $M$ to its underlying abelian group equipped with 
the $(G \hltimes A)$-action given by $m(\chi_g \hltimes a) = m_g a$ for each $g \in G$, $a \in A$, and $m \in M$.
For a homogeneous homomorphism $f$, $\Psi(f)$ is set to be same function $f$; 
a direct computation shows that the latter is $G \hltimes A$-linear. 
Fix a set $\cU$ of graded local units for $A$.
We define the module $\Phi(N)$ for a $G \hltimes A$-module $N$ 
as the same abelian group equipped with the grading
$\Phi(N)_g = \sum_{u \in \cU} N(\chi_g \hltimes u)$ for each $g \in G$. 
The equality $\Phi(N) = \bigoplus_{g \in G} \Phi(N)_g$ follows 
from the fact that $N$ is unital. The $A$-module action is given by 
$ma = m(\chi_g \hltimes a)$ for each $m \in \Phi(N)_g$, $a \in A$.
If $f:M\to N$ is a morphism of $G\hltimes A$-modules, we set $\Phi(f)(m_g) = f(m)_g$.
The categories $ G_{\gr}\cat{-mod}_A$ and $\cat{mod}_{G\hltimes A}$ are equipped with right actions of $G$ which we now describe. For $M\in  G_{\gr}\cat{-mod}_A$ and $g\in G$, let $M[g]$ be the same $A$-module with the following grading
\begin{equation}\label{eq:right-shift}
M[g]_h=M_{gh}.
\end{equation}
Likewise, for $N\in\cat{mod}_{G\hltimes A}$, let $N\cdot g$ be the same abelian group $N$ with right multiplication $\cdot_g$ defined as follows
\begin{equation}\label{eq:cdot-g}
x\cdot_g(\chi_s\hltimes a)=x\cdot(\chi_{gs}\hltimes a).
\end{equation}
It is straightforward to check that $\Psi$ and $\Phi$ intertwine these actions.

By definition
both compositions of $\Psi$ and $\Phi$ yield the respective identity functors.
Moreover, both functors are exact; in particular, 
they preserve projective objects.

\begin{prop} \label{prop:fingen} Let $A$ be a $G$-graded ring with graded local units. Then 
the functors $\Psi$ and $\Phi$ of \eqref{mor:mod-cat-iso} 
send finitely generated modules to finitely generated modules.
\end{prop}
\begin{proof} This is shown in the course of the proof of \cite{preuhyper}*{Proposition 66}.
\end{proof}

\begin{rmk}\label{rem:unitalize} Recall that the \emph{unitalization} of a ring $A$ is the abelian group $\widetilde{A}_\Z = A \oplus \Z$ together with the multiplication rule $(a,k) \cdot (b,l) := (ab+al+bk,kl)$. If $A$ is $G$-graded, then so is $\tilde{A}_\Z$, with homogeneous components $(\tilde{A}_\Z)_1=A_1\oplus\Z$ and $(\tilde{A}_\Z)_g=A_g$ for $g\ne 1$.  When $A$ is unital,  $\widetilde{A}_\Z \to A \times \Z$, $(a,n)\mapsto (a+n\cdot 1,n)$ is a $G$-graded isomorphism. 

Let $\cat{Ring}$ be the category of rings and ring homomorphisms and $\cat{Ring}_1\subset\cat{Ring}$ the subcategory of unital rings and unit preserving homomorphisms. A functor $F \colon G_{\gr}-\cat{Ring}_1 \to \cat{Ab}$ is \emph{additive} if for each pair of 
$G$-graded $R,S \in \cat{Ring}_1$ the canonical map $F(R \times S) \to F(R) \oplus F(S)$ is an isomorphism. 
For such an $F$, the functor $\hat F \colon G_{\gr}-\cat{Ring} \to \cat{Ab}$, $\hat F(A) := \ker(F(\widetilde{A}_\Z) \xto{F(\pi_A)} F(\Z))$ extends $F$ up to natural isomorphism.

Since both unitalization and kernels preserve filtering colimits, if $F$ preserves filtering colimits then 
so does $\hat F$. Thus, given a $G$-graded ring $A$ with a set of (graded) local units $\cU$ 
we have $\hat F(A) \cong \colim_{e \in \cU} \hat F(eAe) \cong \colim_{u \in \cU} F(eAe)$. 
In particular all of this applies to $F(A)=K_*^{G_{\gr}}(A)$ which is defined for unital $A$ as the $K$-theory of the split-exact category $\cat{Proj}_{G_{\gr}}(A)$ of finitely generated projective $G$-graded $A$-modules, and we have 
\begin{equation}\label{eq:kgrlocu}
K_*^{G_{\gr}}(A)=K_*(\cat{Proj}_{G_{\gr}}(A))
\end{equation}
whenever $A$ has graded local units.
\end{rmk}

\begin{thm}\label{thm:k-cross=k-gr} Let $A$ be a $G$-graded ring with graded local units. Then 
the functor $\Psi$ above induces an isomorphism of $\Z$-graded $G$-modules $K_*^{G_{\gr}}(A)\iso K_*(G\hltimes A)$.
\end{thm}
\begin{proof} Let $\cU\subset A$ be a set of graded local units and let $\cF(G)$ be the set of all finite subsets of $G$. Then $\{\chi_F\hltimes e: F\in\cF(G),\, e\in\cU\}$ is a set of local units of $G\hltimes A$. Hence the theorem follows from Proposition \ref{prop:fingen} and Remark \ref{rem:unitalize}.
\end{proof}

\begin{rmk} Let $R$ be a ring with graded local units. Recall from 
Subsection \ref{subsec:cropro} that $G \hltimes R$
is a $G$-$\ast$-algebra; write $\alpha_g \colon G \hltimes R \to G \hltimes R$, $\alpha_g(\chi_s\hltimes a)=\chi_{gs}\hltimes a$,
for the algebra automorphism associated with left multiplication by $g \in G$. The functor 
$\beta_g$ that sends a unital right $G\hltimes R$-module $N$ to  $N \cdot g$  as defined in
\eqref{eq:cdot-g} --which corresponds to shift grading \eqref{eq:right-shift} under the 
equivalence \eqref{mor:mod-cat-iso}--
is naturally isomorphic to scalar extension along $\alpha_{g^{-1}}$. Thus the left and 
right $\Z[G]$-module structures induced by the $\alpha_g$ and the $\beta_g$ on 
$K_*(G\hltimes R)$ correspond to each other under the canonical involution $g\mapsto 
g^{-1}$.
\end{rmk}

\subsection{Duality and hermitian \topdf{$K$}{K}-theory}\label{subsec:dual}
Let $A$ be a $G$-graded $\ast$-ring. If $M$ is a unital $G$-graded right $A$-module, 
its \emph{hermitian dual} is the right module
\[
M^\ast = \{f \in \hom_\Z(M,A) : f(xa) = a^\ast f(x)
\ (\forall a \in A, x\in M)\}.
\] 
For each $d\in G$, consider the $\Z$-submodule
\[
M^\ast_d := \{f \in M^\ast : f(M_g) \subset A_{g^{-1}d} \ (\forall g \in G)\}    
\]
The sum $\sum_{d\in G}M^\ast_d$ is always direct, but the inclusion $\bigoplus_{d\in G}M^\ast_d\subset M^{\ast}$ may be strict. Assume that $A$ has graded local units; we say that a $G$-graded $A$-module $M$ is \emph{finitely presented} if there exist a homogeneous idempotent $e\in A$, $n,m\ge 1$, $g_1,\dots,g_n,h_1,\dots,h_m\in G$ and an exact sequence of homogeneous homomorphisms
\[
\bigoplus_{i=1}^neA[g_i]\to \bigoplus_{i=1}^meA[h_i]\to M\to 0. 
\]

If $M$ is finitely presented, then we have the equality
\[
M^\ast=\bigoplus_{d\in G}M^\ast_d.
\]
Hence $M^\ast$ is a $G$-graded module in this case. 

\begin{lem} \label{lem:psi-*=*-psi}
Let $A$ be a $G$-graded $\ast$-ring with self-adjoint graded local units and $M$ a finitely presented graded unital right $A$-module. There is a natural isomorphism $\eta_M \colon \Psi(M^\ast) \iso \Psi(M)^\ast$.
\end{lem}
\begin{proof} As in Subsection \ref{subsec:cropro}, we regard $G\hltimes A$ as $*$-subalgebra of $M_GA$; we write $\pi_{g,h}(x)\in A$ for the $(g,h)$-entry of $x\in G\hltimes A$.
For $\alpha \in \Psi(M^\ast)$ and $\beta \in \Psi(M)^\ast$, set
\[
\alpha^\sharp (x) = \sum_{g,h \in G} \chi_g \ltimes \alpha_h(x_g), 
\quad \beta^\flat(y) = \sum_{g,h \in G} \pi_{g,h}(\beta(y))   
\]
One checks that $\alpha^\sharp \in \Psi(M)^\ast$ and $\beta^\flat \in \Psi(M^\ast)$ and that
the map $(-)^\sharp \colon \Psi(M^\ast) \iso \Psi(M)^\ast$ is a $G \hltimes A$-linear 
isomorphism with inverse $(-)^\flat$.
\end{proof}

Let $A$ be a $G$-graded $*$-ring with graded local units. If $M$ is a unital, finitely generated and projective $G$-graded right $A$-module, then the canonical map 
\begin{equation}\label{map:candualdual}
\can:M\to M^{\ast\ast}, \quad \can(x)(f) = f(x)^\ast
\end{equation}is a homogeneous isomorphism. Hence the triple $(\cat{Proj}_{G_{\gr}}(A),\ast,\can)$ with the split-exact 
sequences as conflations, is an exact category with duality 
in the sense of \cite{marcoherm}*{Definition 2.1}. If $A$ is unital, we write 
\begin{equation}\label{khgr}
K_*^{h,G_{\gr}}(A)=GW_*(\cat{Proj}_{G_{\gr}}(A),\ast,\can)
\end{equation}
for its Grothendieck-Witt groups. We extend $K_*^{h,G_{\gr}}$ to all, not
necessarily unital $G$-graded $\ast$-rings as in Remark \ref{rem:unitalize}.
If $A$ is a $G$-graded $\ast$-ring with self-adjoint graded units, then it
can be written as a filtering colimit of unital $G$-graded rings with
involution and so its $G$-graded hermitian $K$-theory groups coincide with
the Grothendieck-Witt groups of $(\cat{Proj}_{G_{\gr}}(A),\ast,\can)$.
Similarly, the hermitian $K$-theory groups of $G\hltimes A$ are the
Grothendieck-Witt groups of the split exact category with duality
$(\cat{Proj}(G\hltimes A),\ast,\can')$. Here $\can'$ is the natural
transformation between a $G\hltimes A$-module and its double hermitian dual. 

\begin{lem} \label{lem:psi-form-fun} Let $A$ be a $G$-graded
$\ast$-ring with self-adjoint graded local units. Let $\can$ and $\can'$ be as above. Then for each unital $G$-graded $A$-module $M$ the following diagram is commutative.
\begin{equation*}
\begin{tikzcd}
 \Psi(M) \arrow{r}{\can'_{\Psi(M)}} \arrow{d}[left]{\Psi(\can_M)} & \Psi(M)^{\ast\ast} \arrow{d}{\eta^\ast_{M}} \\
 \Psi(M^{\ast\ast}) \arrow{r}[below]{\eta_{M^\ast}} & \Psi(M^{\ast})^{\ast}
\end{tikzcd}
\end{equation*}\
In other words, $(\Psi, \eta)$ is a form functor $(\cat{Proj}_{G_{\gr}}(A),\ast,\can)\to (\cat{Proj}(G\hltimes A),\ast,\can')$ in the sense of \cite{marcoherm}*{Definition 3.2}.
\end{lem}
\begin{proof} It suffices to see that given an element $m \in \Psi(M)$ which 
is homogeneous as an element of $M$, the maps $(\eta^\ast_M\can'_{\Psi(M)})(m)$ and $(\eta_{M^\ast}\Psi(\can))(m)$ coincide when evaluating them at each element $\alpha \in \Psi(M^\ast)$ which is homogeneous as an element of $M^\ast$. 
Indeed, given $m \in M_s$ and $\alpha \in M^\ast_t$ we have
\[
(\eta^\ast_M\can'_{\Psi(M)})(m)(\alpha) = \can'_{\Psi(M)}(m)\eta_M(\alpha) = (\alpha^\sharp(m))^\ast
= (\chi_s \hltimes \alpha(m))^\ast = \chi_t \hltimes \alpha(m)^\ast 
\]
and 
\[
(\eta_{M^\ast}\Psi(\can))(m)(\alpha) = \Psi(\can)(m)^\sharp (\alpha) =
\chi_t \hltimes \Psi(\can)(m)(\alpha) = \chi_t \hltimes \alpha(m)^\ast.
\]
\end{proof}

\begin{rmk}\label{rmk:form-fun-stric} 
Let $\cat A$ and $\cat B$ be exact categories with duality. An exact form functor 
$(F,\varphi) \colon \cat A \to \cat B$ is nonsingular if $\varphi$ is an isomorphism, in which case it induces a homomorphism between the 
Grothendieck-Witt groups of $\cat A$ and $\cat B$ \cite{marcoherm}*{Section 3.1}. If $F$ is moreover 
an isomorphism, then by definition of
form functor composition \cite{marcoherm}*{Definition 3.2} 
the (strict) inverse $F^{-1}$ can be equipped with 
a non-singular form functor structure  in such a way that
$F$ and $F^{-1}$ induce inverse isomorphisms on Grothendieck-Witt
groups  as 
defined in \cite{marcoherm}*{Definition 4.12}.  
\end{rmk}

\begin{thm}\label{thm:kh-cross=kh-gr} Let $A$ be a $G$-graded $\ast$-ring with self-adjoint graded local units. Then
the morphisms \eqref{mor:mod-cat-iso} induce an isomorphism of graded $\Z[G]$-modules
\begin{equation}\label{map:isokh}
K_*^{h,G_{\gr}}(A) \iso K^h_*(G \hltimes A).
\end{equation}
\end{thm}
\begin{proof}
A colimit argument similar to that used to show \eqref{eq:kgrlocu} proves that the left and right hand side of \eqref{map:isokh} are the Grothendieck-Witt groups of the exact categories with duality $(\cat{Proj}_{G_{\gr}}(A),\ast,\can)$ and $(\cat{Proj}(G\hltimes A),\ast,\can')$. The functor $\Psi$ of \eqref{mor:mod-cat-iso} induces an isomorphism between these Grothendieck-Witt groups, by Proposition \ref{prop:fingen}, Lemmas \ref{lem:psi-*=*-psi} and \ref{lem:psi-form-fun} and Remark \ref{rmk:form-fun-stric}.
\end{proof}
\begin{coro}\label{coro:k-cross=k-gr}
Let $A$ be a $G$-graded $\ast$-ring and let $\tilde{A}_\Z$ be its unitalization. 
Then for all $n\in\Z$ we have
\[
K_n^{h,G_{\gr}}(A)=\ker(K^h_n(G\hltimes\tilde{A}_\Z)\to K^h_n(\Z^{(G)})).
\]
If furthermore $n\le 0$, then  $K_n^{h,G_{\gr}}(A)=K^h_n(G\hltimes A)$.  
\end{coro}
\begin{proof} The general formula follows from Theorem \ref{thm:kh-cross=kh-gr} and the fact that 

$$K_n^{h,G_{\gr}}(A)=\ker(K_n^{h,G_{\gr}}(\tilde{A}_\Z)\to K_n^{h,G_{\gr}}(\Z)).$$

The second assertion is a consequence of the first and of the fact that
hermitian $K$-theory satisfies excision in non-positive dimensions. 
\end{proof}

\subsection{Homotopy hermitian graded \topdf{$K$}{K}-theory}\label{subsec:khhgr}
Let $A$ be a $G$-graded $*$-algebra, $PA=\ker(\ev_0:A[t]\to A)$, 
$\Omega A=\ker(\ev_1:PA\to A)$. By Corollary \ref{coro:k-cross=k-gr},  
hermitian graded $K$-theory satisfies excision in nonpositive degrees. In 
particular the path extension 
\[
0\to \Omega A\to PA\overset{\ev_1}{\lra} A\to 0
\]
gives rise to a connecting map 
$K_{n}^{h,G_{\gr}}(A)\to K_{n-1}^{h,G_{\gr}}(\Omega A)$ for each $n \leq 0$. Set 
\begin{equation}\label{def:khgr}
KH^{h,G_{\gr}}_n(A)=\colim_{r \ge 0} K_{-r}^{h,G_{\gr}}(\Omega^{n+r}A).
\end{equation}
It follows from Corollary \ref{coro:k-cross=k-gr} and the definition of $KH^h$ \cite{kkh}*{Section 3} that 
\begin{equation}\label{eq:khgr}
KH^{h,G_{\gr}}_n(A)=KH_n^h(G\hltimes A).
\end{equation}
for every $G$-graded $\ast$-ring $A$ and every $n\in\Z$.

\subsection{Free involutions}\label{subsec:gwinv}
If $R$ is a $G$-graded ring, the ring $\inv(R) =R\oplus R^{\op}$ equipped with the involution $(a,b)^*=(b,a)$ has a compatible $G$-grading, with $\inv(R)_g= R_g\oplus R^{\op}_g = R_g\oplus R_{g^{-1}}$. There is an isomorphism $G \hltimes \inv(R) \iso \inv (G\hltimes R)$, 
$\chi_g \hltimes (x \otimes r)  \mapsto x\otimes (\chi_g \hltimes r)$. 

\begin{prop} \label{prop:kgr=khgr-inv}
Let $R$ be a $G$-graded unital ring. There is a natural isomorphism
$K^{G_{\gr}}_\ast(R) \cong K^{h,G_{\gr}}_\ast(\inv(R))$.
\end{prop}
\begin{proof} Via the isomorphism 
$\cat{Proj}_{G_{\gr}}(\inv(R)) \cong \cat{Proj}_{G_{\gr}}(R) 
\times \cat{Proj}_{G_{\gr}}(R^{\op})$,
the duality functor of $\cat{Proj}_{G_{\gr}}(\inv(R))$
can be naturally identified with the endofunctor of 
$\cat{Proj}_{G_{gr}}(R) \times \cat{Proj}_{G_{gr}}(R^{\op})$
given by $(P,Q)^\dagger = (Q^\vee, P^\vee)$ where $(-)^\vee$ 
denotes the non-hermitian dual. Writing $\ev_M \colon M \to M^{\vee \vee}$, 
$\ev_M(x)(f) = f(x)$, 
the natural transformation 
$\can \colon \cat{Proj}_{G_{\gr}}(\inv(R)) \to \cat{Proj}_{G_{\gr}}(\inv(R))$
is identified with $\ev \times \ev$.

In view of Remark \ref{rmk:form-fun-stric} and 
\cite{marcoherm}*{Proposition 4.7}, it suffices to see that
the Grothendieck-Witt groups 
of $(\cat{Proj}_{G_{\gr}}(R) \times \cat{Proj}_{G_{\gr}}(R^{\op}), \dagger, \ev \times \ev)$
coincide with those of the the hyperbolic category $H(\cat{Proj}_{G_{\gr}(R)})$ of 
\cite{marcoherm}*{Section 3.5}. The latter consists of the category
$\cat{Proj}_{G_{\gr}}(R) \times \cat{Proj}_{G_{\gr}}(R)^{\op}$ together with 
the duality functor  $(P,Q)^\ast = (Q,P)$ 
and the identity natural transformation $\cat{id} : 1 \Rightarrow \ast\ast$.
We now consider the inverse equivalences 
\[
    F := \cat{id} \times (-)^\vee \colon
    \cat{Proj}_{G_{gr}}(R) \oplus \cat{Proj}_{G_{gr}}(R^{\op}) \longleftrightarrow
    \cat{Proj}_{G_{gr}}(R) \oplus \cat{Proj}_{G_{gr}}(R)^{\op} 
    \colon \cat{id} \times (-)^\vee =: G. 
\]
We promote these inverse equivalence to non-singular form functors by means of the natural
transformations
\[
\varphi \colon F\circ \dagger \Rightarrow \ast\circ F, 
\quad \varphi_{(P,Q)} = (1_{Q^\vee}, \ev_P)
\colon (Q^\vee, P^{\vee\vee}) \to (Q^\vee,P)
\]
and
\[
\psi \colon G\circ \ast \Rightarrow \dagger\circ G, 
\quad \psi_{(P,Q)} = (\ev_Q,1_{P^\vee})
\colon (Q,P^\vee) \to (Q^{\vee\vee},P^\vee).
\]
Let $\star$ be as in \cite{marcoherm}*{bottom of page 113}; the form functor compositions $(F \circ G, \varphi \star \psi)$ and $(G \circ F, \psi \star \varphi)$
yield
\begin{align*}
F \circ G &= \id \times (-)^{\vee\vee}, \quad (\varphi \star \psi)_{(P,Q)} = (\ev_Q, \ev_P),\\
G \circ F &= \id \times (-)^{\vee\vee}, \quad (\psi \star \varphi)_{(P,Q)} = (\ev_{Q^\vee}, {\ev_P}^\vee).
\end{align*}
One checks that
$\zeta \colon F\circ G \Rightarrow \id$, $\zeta_{(P,Q)} = (1_P,\ev_Q)$
and $\xi \colon \id \Rightarrow G \circ F$, $\xi_{(P,Q)} = (1_P,\ev_Q)$
are natural isomorphisms of form functors as defined in \cite{marcoherm}*{Section 3.3}.
This implies that the morphisms induced by $F$ and $G$ at the level of 
Grothendieck-Witt groups are mutually inverse; see \cite{marcoMV}*{Section 2.8, Lemma 2} and \cite{marcoMV}*{Section 2.10, Proposition 2}.
\end{proof}

\begin{rmk} By Proposition \ref{prop:kgr=khgr-inv}, the isomorphism
$K_*^{G_{\gr}}(A) \cong K_*(G\hltimes A)$ of Theorem \ref{thm:k-cross=k-gr}
can be recovered from Theorem \ref{thm:kh-cross=kh-gr} applied to $\inv(R)$.
\end{rmk}

\section{Hermitian Dade theorem}\label{sec:hdade}
\numberwithin{equation}{section}
Recall that a $G$-graded ring 
$R$ is \emph{strongly graded} if for every $g,h\in G$,
we have $R_g\cdot R_h=R_{gh}$. Dade proves in \cite{dade}*{Theorem 2.8} that 
a unital ring $R$ is strongly graded if and only if the functors 
\begin{equation}\label{fun:dade}
    (-)_1 \colon   G_{\gr}-\cat{mod}_R \to \cat{mod}_{R_1} \qquad \text{and} \qquad 
    - \otimes_{R_1} R \colon \cat{mod}_{R_1} \to  G_{\gr}-\cat{mod}_R 
\end{equation}
are mutually inverse category equivalences. 

The categories above, equipped with the natural transformation $\can_M \colon M \to M^{\ast\ast}$ are 
categories with duality as defined in \cite{marcoherm}*{Definition 3.1}. 
Similarly, the functors \eqref{fun:dade} equipped with the following transformations are form functors in the sense of \cite{marcoherm}*{Definition 3.2}.
\begin{align}\label{nat:dadeform}
&\varphi_M \colon (M^*)_1 \mapsto (M_1)^\ast, 
\quad \varphi_M(f) = f|_{M_1}^{R_1};\\\nonumber
&\psi_N \colon N^\ast \otimes_{R_1} R \to (N \otimes_{R_1} R)^\ast, 
\quad \psi_N(f \otimes r)(n \otimes s) = s^\ast f(m)  r. 
\end{align}
With these definitions in place, we have 
the following hermitian version of Dade's theorem.

\begin{thm}\label{thm:h-dade}
Let $R$ be a unital $G$-graded $\ast$-ring. The following are equivalent:
\begin{itemize}
    \item[(i)] The ring $R$ is strongly graded. 
    \item[(ii)] The form functors given by \eqref{fun:dade} and 
    \eqref{nat:dadeform} are mutual inverse equivalences of 
    categories with duality.  
\end{itemize}
\end{thm}
\begin{proof} By the non-hermitian version of the theorem the functors $(-)_1$ and $- \otimes_{R_1} R$ are 
mutual inverses if and only $R$ is strongly graded; see e.g. 
\cite{hazbook}*{proof of Theorem 1.5.1}. A straightforward computation 
shows that when $R$ is strongly graded the natural isomorphisms 
$(N \otimes_{R_1} R)_1 \cong N$ and $M_1 \otimes_{R_1} R \cong M$ given
in loc. cit. are in fact natural isomorphisms of form functors. 
\end{proof}

\begin{rmk} \label{rmk:dade-h-cross} 
Let $R$ be as in Theorem \ref{thm:h-dade}. The composite
$\Psi \circ (-\otimes_{R_1} R) 
\colon \cat{Proj}(R_1) \to \cat{Proj}(G \hltimes R)$ 
is naturally isomorphic to the functor 
induced by the ring morphism 
$r \in R_1 \mapsto \chi_1 \hltimes r \in G \hltimes R$
by means of the natural isomorphism
\[
\eta_M \colon M \otimes_{R_1} R \to M \otimes_{R_1}( G \hltimes R), 
\qquad m \otimes r \mapsto m \otimes \chi_1 \hltimes r.
\]
\end{rmk}

\begin{ex}\label{ex:groupalg}
Let $R,S\in G_{\gr}-\ahas$; if $R$ is either trivially graded or $G$ is abelian, then by Remark \ref{rmk:tensor-gr}, $R\otimes S$ is again $G$-graded. If moreover $R$ and $S$ are unital, and $S$ is strongly graded, then so is $R\otimes S$. In particular this applies to $S=\ell[G]$, so $R[G]$ is strongly graded. Moreover, $R[G]_1=\bigoplus_{g\in G}R_g\otimes g^{-1}\cong R$, so by Theorem \ref{thm:h-dade}, we have
\[
K^{h,G_{\gr}}_*(R[G])=K^h_*(R). 
\]
\end{ex}

\section{Graded \topdf{$K$}{K}-theory of Leavitt path algebras}\label{sec:kleav}

Let $E$ be a (directed) graph and $L(E)$ its Leavitt path algebra over $\ell$ (\cite{lpabook}). In this section we show that if $E$ is row-finite, then
its Bowen-Franks $\Z[\sigma]$-module \eqref{eq:bfgr}
together with the (hermitian) $K$-theory of $\ell$ completely 
characterize the graded $K$-theory of $L(E)$.
We write 
\[
\sink(E) = \{v \in F^0 : s^{-1}(v) = \emptyset\},
\]
for the sets of sinks of a graph $E$. Observe that if $E$ is row-finite, we have $E^0=\reg(E)\sqcup\sink(E)$.

Recall that an exact sequence of abelian groups is \emph{pure exact} if it remains exact upon tensoring with any abelian group.
\begin{lem} \label{lem:mononotepi} Let $E$ be a row-finite graph. Then
\item[i)] The sequence of abelian groups
\begin{center}
    \begin{tikzcd}
        0 \arrow{r} & \Z[\sigma]^{(\reg(E))} \arrow[above]{r}{I-\sigma A_E^t} &
        \Z[\sigma]^{(E^0)} \arrow{r} & \gBF(E) \arrow{r} & 0
    \end{tikzcd}
\end{center}
is pure exact.
\item[ii)] If $E$ is finite, then $\gBF(E)\ne 0$.
\end{lem}
\begin{proof} 
The sequence of part i) is pure exact because it is the colimit over $n$ of 
the split-exact sequences
\[
    \begin{tikzcd}
        0 \to \bigoplus_{i=-n}^n\Z^{(\reg(E))}\sigma^i\arrow{r}{I-\sigma A_E^t} &
        \bigoplus_{i=-n}^{n+1}\Z^{(E^0)} \arrow{r} &
        \left(\bigoplus_{i=-n}^{n+1}\Z^{(\sink(E))}\sigma^i\right) \oplus 
        \Z^{(\reg(E))}\sigma^{n+1} \to 0.
    \end{tikzcd}        
\]
Next assume that $I-\sigma \cdot A_E^t$ is surjective and let $\chi_{A_E}(\sigma)=\det(\sigma I-A_E^t)\in\Z[\sigma]$ be the characteristic polynomial.
By rank considerations $E$ must be regular, and by what we have just proved, 
 $\det(I-\sigma A_E^t) = \sigma^{|E^0|}\cdot\chi_{A_E}(\sigma^{-1})$, must be invertible in $\Z[\sigma]$. It follows that $\chi_{A_E}(\sigma)$ is also invertible, and therefore it is a power of $\sigma$; by degree considerations we must have 
$\chi_{A_E}(\sigma)=\sigma^{|E^0|}$. In particular $A_E$ is nilpotent, contradicting 
the regularity of $E$. 
\end{proof}

In the next theorem we write $\ucov{E}$ for the covering associated to the constant function $\omega:E^1\to\Z$, $\omega(e)=1$ for all $e\in E^1$. 
\begin{thm}\label{thm:bf-k0=k0gr} 
Let $E$ be a row-finite graph and let $H\colon \ahas \to \cat{Ab}$ be an $M_{\infty}$-stable, additive functor that preserves filtering colimits. Then there is a $\Z[\sigma]$-module isomorphism
\[
H(L(\ucov{E}))\cong \gBF(E) \otimes_\Z H(\ell).
\]
\end{thm} 
\begin{proof} 
Because every row-finite graph is the filtering colimit of its finite complete subgraphs \cite{amp}*{Lemma 3.1}, we may assume that $E$ is finite. The canonical action of the generator $\sigma$ of $\Z$ on $\ucov{E}$ induces a $\ast$-algebra automorphism $s \colon L(\ucov{E}) \to L(\ucov{E})$ such that $s(v_n) = v_{n+1}$, $s(e_n) = e_{n+1} \, (v \in E^0, e \in E^1)$. For $n \in \N_0$, let $E_n$ be the full subgraph of $\ucov{E}$
on the vertices $E^0_n = \{v_k : v \in E^0, \ |k| \leq n\}$ and observe that
$s$ can be (co-)restricted to a map $s_n \colon L(E_{n}) \to L(E_{n+1})$ for all $n \geq 0$.
Since $\ucov{E} = \bigcup_{n \geq 1} E_n$, we have $L(\ucov{E}) = \colim_{n \geq 0} L(E_n)$ and thus
\[
H(L(\ucov{E})) \cong \colim_{n \geq 0} H(L(E_n)).
\]
The action induced by $s$ carries over to the colimit by means of the automorphism associated 
to the family $\{H(s_n)\}_{n \geq 0}$. 
Observe that each graph $E_n$ is acyclic. We shall use the fact that for an acyclic graph $F$ and the set $\cP_v(F)$ of paths ending in $v\in\sink(F)$, there is a $*$-isomorphism $L(F)\cong \bigoplus_{v\in\sink(F)}M_{\cP_v(F)}$. This is proved in \cite{lpabook}*{Theorem 2.6.17} under the assumption that $\ell$ is a field, but the proof does not use it; see also \cite{alg2.5}*{Proposici\'on 2.5.1}. By additivity and $M_\infty$-stability of $H$, we have an
isomorphism $H(L(E_n))\cong\Z^{\sink(E_n)}\otimes_\Z H(\ell)$. Observe that 
\[
\sink(E_n)=(\sink(E)\times\{i\in\Z: |i|\le n\})\sqcup( \reg(E)\times\{n\}).
\]
The transition map $H(\ell)\otimes\Z^{\sink(E_n)}\to H(\ell)\otimes\Z^{\sink(E_{n+1})}$ is induced by 
the inclusion $\Z^{\sink(E)\times\{i\in\Z:|i|\le n\}}\subset \Z^{\sink(E)\times\{i\in\Z:|i|\le n+1\}}$ and by 
\begin{equation}\label{eq:transmap}
(v,n)\mapsto \sum_{w\in r(s^{-1}\{v\})}A_E(v,w)(w,n+1)
\end{equation}
on $\Z^{\reg(E)\times\{n\}}$. It follows that  $\colim_{\N_0} \Z^{\sink(E_n)} \otimes_\Z H(\ell)
\cong \gBF(E) \otimes_\Z H(\ell)$.
\end{proof}

In what follows we write $K^{\gr}$ and $K^{h,\gr}$ for $K^{\Z_{\gr}}$ and $K^{h,\Z_{\gr}}$.

\begin{coro}\label{coro:bf-k0=k0gr} Let $E$ be a row-finite graph and $R$ a $\ast$-algebra. If $R$ has local units and trivial $\Z$-grading, then there are $\Z[\sigma]$-module 
isomorphisms 
\[
K_n^{{\gr}}(L(E)\otimes R) \cong \gBF(E) \otimes_\Z K_n(R), \quad 
K_n^{h,{\gr}}(L(E)\otimes R) \cong \gBF(E) \otimes_\Z K^h_n(R)
\]
\end{coro}
\begin{proof} We have isomorphisms $K_n^{{\gr}}(L(E)\otimes R) \cong K_n(L(\ucov{E})\otimes R)$ and
$K_n^{h,{\gr}}(L(E)\otimes R) \cong K^h_n(L(\ucov{E})\otimes R)$ 
by Proposition \ref{prop:crosscover} and Theorems \ref{thm:k-cross=k-gr} and \ref{thm:kh-cross=kh-gr}. 
The result now follows from Theorem \ref{thm:bf-k0=k0gr}. 
\end{proof}

\begin{lem} \label{lem:vertexlc-nonzero} 
Let $E$ be a finite graph and $\cl:\Z^{E^0}\to \gBF(E)$, $\cl(v)=[v]$, the canonical map. We have $\ker(\cl)\cap\N^{E^0}=0$. 
\end{lem}
\begin{proof} Write $\gBF(E) \cong
\colim_{\N} \BF(E_n) \cong \colim_{\N} \Z^{\sink(E_n)}$ 
as in the proof of Theorem \ref{thm:bf-k0=k0gr}. 
The transition maps $\cl_n:\Z^{\sink(E_n)}\to \Z^{\sink(E_{n+1})}$ are as described in the proof of Theorem \ref{thm:bf-k0=k0gr}; in particular they are given by matrices with nonnegative coefficients and no zero columns. Hence $\cl_n(\N^{\sink(E_n)})\subset \N^{\sink(E_{n+1})}$ and $\N^{\sink(E_n)}\cap\ker(\cl_n)=0$. It follows that $\ker(\cl)\cap \N^{E^0}=0$.
\end{proof}

\begin{thm}\label{thm:reflinj}
Let $E$ be a finite graph and $f \colon L(E) \to R$
a $\Z$-graded algebra homomorphism. Assume that $\ell$ is a field. If $K_0^{\gr}(f)$ is a monomorphism, then $f$ is a monomorphism.
\end{thm}
\begin{proof} Let $v \in E^0$. Because $\gBF(E)\owns [v] \neq 0$ 
by Lemma \ref{lem:vertexlc-nonzero}, the fact that $K_0^{\gr}(f)$ is a monomorphism 
implies that $[f(v)] = K_0^{\gr}(f)([v])$ is nonzero. In particular, we have $f(v) \neq 0$
for all $v \in E^0$, and so the result follows from the graded uniqueness theorem 
for Leavitt path algebras (see e.g. \cite{lpabook}*{Theorem 2.2.15}).
\end{proof}

\section{\topdf{$G$}{G}-Stability} \label{sec:stab}
\begin{stan} From now on, we will
assume the existence of an element $\lambda \in \ell$ satisfying \eqref{eq:lambda}. 
\end{stan}
\begin{conv}\label{conv:gstab} Fix an infinite set $\fX$ of cardinality 
greater or equal than that of $G\sqcup\N$. We will write 
$\fS$ for the categories of $G$-sets or $G$-graded sets of 
cardinality less or equal than $\#\fX$, and $|-|:\fS\to \mathrm{Sets}$ for the forgetful functor. 
If $X\in\fS$ and $A\in\AnyAlg$ we equip $M_XA$ with the $G$-action or grading induced by the inclusion into $C_X A$ or $C_X^\circ A$. If $X$ is just a set, then unless specified otherwise, we will regard it as an object of $\fS$ with the trivial $G$-action or grading.
\end{conv}

We say that a functor is \emph{matricially stable} to mean that it is $M_{\fX}$-stable
as defined in Subsection \ref{subsec:matstab}.

\begin{rmk} If $R$ is a unital $\ast$-algebra and $X$ is any set, 
then both $C_X R$ and $\Gamma_X R$ are also unital with unit 
$1 = \sum_{x \in X} \elmat_{x,x}$. Note that if $R$ is $G$-graded
then this element belongs to $\Gamma_X^\circ R\subset C_X^\circ R$,
hence both of these $\ast$-algebras are unital as well.
\end{rmk}

We define an \emph{ideal embedding} in $\AnyAlg$ to be
a monomorphism $i \colon A \to R$ such that 
$R$ is unital and $i(A)$ is a $G$-invariant (resp. homogenous) $\ast$-ideal of $R$. 
Let $R$ be a unital algebra in $\AnyAlg$ and $X$ in $\fS$. 
In the equivariant case, we set $C_X^\circ R = C_X R$. 
Let $\eps$ be a central unitary; in the equivariant case, assume $\eps$ is fixed by $G$; in the $G$-graded case, assume it is homogeneous of degree $1$. An element $\phi \in C_X^\circ R$ is \emph{$\eps$-invariant} if 
it is $\eps$-hermitian and fixed by $G$ 
(resp. and homogeneous of degree $1$). We are 
now in position to define the notion of $G$-stability.

\begin{defn} A functor $F \colon \AnyAlg \to\cat{C}$ is \emph{$G$-stable} 
if for any pair of sets $X,Y \in \fS$, 
any ideal embedding $i \colon A \to R$ and any two invariant $\eps$-hermitian 
elements $\phi \in C_X^\circ R, \psi \in C_Y^\circ R$, $F$ sends the inclusion 
\[
(M_X A)^\phi \lra (M_{X \sqcup Y} A)^{\phi \oplus \psi}
\]
to an isomorphism.
\end{defn}

Put 
\[
h_\pm := \begin{pmatrix}1 & 0\\0 & -1\end{pmatrix}.
\]
This element is $1$-hermitian; write $M_\pm := M_2^{h_\pm}$. 
For each $A$ in $\AnyAlg$, we write $\iota_+\colon A \to M_\pm A$ 
for the upper-left corner inclusion. A functor $F \colon \AnyAlg \to \cat{C}$ 
is \emph{$\iota_+$-stable} if the natural transformation $F(\iota_+):F\to F\circ M_{\pm}$ is an isomorphism.

The following proposition says that 
under the hypothesis of $\iota_+$-stability one may 
omit considering the $\eps$-invariant elements $\phi$ and $\psi$ in the definition above.

\begin{prop}\label{prop:G-stab-no-her} An $\iota_+$-stable functor
$F \colon \AnyAlg \to \cat{C}$ is $G$-stable
if and only if for each $A$ in $\AnyAlg$ it sends each inclusion
\[
    M_X A \lra M_{X \sqcup Y} A
\]
induced by a pair of sets $X, Y \in \fS$ to an isomorphism.
\end{prop}
\begin{proof} The argument of 
\cite{kkh}*{Proposition 2.4.4} shows this.
\end{proof}

The proof of the following lemma is a straightforward verification.

\begin{lem}\label{lema:triv-ac-gr} 
Let $A$ a $G$-$\ast$-algebra and $B$ a $G$-graded $\ast$-algebra.
\begin{itemize}
    \item[(i)] If $X$ is a $G$-set, then the linear map
    \begin{equation}\label{mor:triv-ac}
    M_G M_X A\owns\elmat_{s,t} \otimes \elmat_{x,y} \otimes a
    \mapsto \elmat_{s^{-1}x, t^{-1}y} \otimes \elmat_{s,t} \otimes a\in M_{|X|} M_G A
    \end{equation}
    is a $G$-$\ast$-algebra isomorphism, natural with respect
to both $X$ and $A$. 
    \item [(ii)] If $X$ is a $G$-graded set, then the linear map 
    \begin{equation}\label{mor:triv-gr}
M_{G}M_X B \owns \elmat_{s,t} \otimes \elmat_{x,y} \otimes b \mapsto
    \elmat_{x,y} \otimes \elmat_{s|x|,t|y|} \otimes b\in M_{|X|}M_G B
    \end{equation}
    is a $G$-graded $\ast$-algebra isomorphism, natural with respect to both $X$ and $B$.
\end{itemize}
\qed
\end{lem}

As a consequence of Lemma \ref{lema:triv-ac-gr}, we obtain the following hermitian analogue of
\cite{kkg}*{Proposition 3.1.9}, which will useful for constructing the equivariant and graded versions of hermitian
bivariant $K$-theory.

\begin{prop}
\label{prop:stab} Let $F : \AnyAlg \to \cat{C}$ be a matricially stable, 
$\iota_+$-stable functor. Then $F(M_G(-))$ has the same stability properties
and is moreover $G$-stable.
\qed
\end{prop}

\begin{rmk}\label{rmk:natisos} Consider the $G$-set $G \sqcup \star$, where $G$ 
acts trivially on $\star$. For each $G$-$\ast$-algebra $A$, 
the inclusions $\star \subset G \sqcup \star \supset G$ yield a 
zig-zag of inclusions
\begin{equation}\label{mor:zigzag-MGA}
    A \to M_{G \sqcup \star} A \leftarrow M_G A
\end{equation}
relating $A$ and $M_G A$. 
Observe that any $G$-stable functor maps \eqref{mor:zigzag-MGA} 
to an isomorphism.

In the graded setting, for each $G$-graded
$\ast$-algebra $B$ the $G$-graded set inclusion $\{1\} \subset G$ 
yields a $G$-graded $\ast$-algebra inclusion 
\begin{equation}\label{mor:iota1}
  \iota_1 \colon B \to M_G B,
\end{equation}
which any $G$-stable functor sends to an isomorphism.
\end{rmk}

\begin{coro} Let $F : \AnyAlg \to \cat{C}$ be a matricially stable, $\iota_+$-stable
functor. The following statements are equivalent:
\begin{itemize}
    \item[(i)] The functor $F$ is $G$-stable.
    \item[(ii)] For each $A$ in $\AnyAlg$ the functor $F$ sends the 
    morphism \eqref{mor:iota1} (resp. the zig-zag \eqref{mor:zigzag-MGA})
    to an isomorphism.
    \item[(iii)] The functors $F$ and $F\circ M_G$ are naturally isomorphic.
\end{itemize}
\end{coro}
\begin{proof} The implication (i)$\Rightarrow$(ii) follows from
Remark \ref{rmk:natisos}. The fact that (ii)$\Rightarrow$(iii) is trivial. The implication (iii)$\Rightarrow$(i) is a consequence of 
Proposition \ref{prop:stab}, together with the fact that $G$-stability 
is preserved under natural isomorphisms. 
\end{proof}

\section{Hermitian graded bivariant \topdf{$K$}{K}-theory}
\label{sec:kkhgr}
\numberwithin{equation}{section}
In this section we introduce bivariant $K$-theory categories
for $\GAlg$ and $\GgrAlg$. For this purpose we first adapt 
the notions of polynomial homotopy invariance, extensions
and excisive homology theories from \cite{kkh} and \cite{kkg} to the present context.

Let $A$ be an object of $\AnyAlg$. We view the polynomial ring
$A[t] = A \otimes_\Z \Z[t]$ as an object of $\AnyAlg$ 
by considering the trivial action or grading on $t$.
A functor $F \colon \AnyAlg \to \cat{C}$ is \emph{(polynomially)
homotopy invariant} if it sends each inclusion 
$A \hookrightarrow A[t]$ of an algebra into its 
ring of polynomials to an isomorphism.

An \emph{extension} in $\AnyAlg$ is an exact sequence
\begin{equation}\label{ext}
    0 \to A \xto{i} B \xto{p} C \to 0.
\end{equation}    
We say that the category of $\ell[G]$-modules or $G$-graded
$\ell$-modules is the \emph{underlying category} of $\fA$ and 
we denote it $\fU$; writing $U \colon \fA \to \fU$ for the canonical forgetful functor, 
we say that \eqref{ext} is \emph{weakly split} if $U(p)$ has a section. 

Write $\cE$ for the class of 
weakly split extensions of $\fA$.
Let $\cat{T}$ be a triangulated category with inverse suspension 
functor $\Omega$. An \emph{excisive homology theory} on 
$\fA$ with values in $\cat{T}$ 
is a functor $H \colon \fA \to \cat{T}$ together with a 
family of maps $\{\partial_E \colon \Omega H(C) \to H(A)\}_{E \in \cE}$
compatible with  morphisms of extensions (\cite{kk}*{Section 6.6})
such that for each $E \in \cE$ there is a distinguished triangle
\[
    \Omega H(C) \xto{\partial_E} H(A) \to H(B) \to H(C).
\]

Applying the analogue of the construction of \cite{kkg}*{Section 2}
to bivariant hermitian $K$-theory \cite{kkh}, we obtain
an excisive homology theory 
\[
    j^h_\fA \colon \fA \to kk^h_\fA
\]
that is $M_\fX$-stable, $\iota_+$-stable and homotopy invariant.
Moreover, any other excisive homology theory $H \colon \fA \to \cat{T}$ 
which satisfies these stability properties factors uniquely
trough $j^h_\fA$ via a triangle functor.

\begin{defn}
We define \emph{equivariant bivariant hermitian $K$-theory} as the category whose 
objects are $G$-$\ast$-algebras and the morphisms are given by
\[
    kk_G^h(A,B) := kk^h_{\GAlg}(M_G A, M_G B),
\]
with the composition rule of $kk^h_{\GAlg}$. There is a canonical functor
\[
    j^h_G \colon \GAlg \to kk^h_G
\]
defined as the identity on objects, and that sends an equivariant map
$f \colon A \to B$ to the image of
$M_G f : M_G A \to M_G B$ via $j^h_{\GAlg}$, 
viewed as an arrow of $kk^h_{G}$.

Likewise, we define \emph{$G$-graded hermitian bivariant $K$-theory} as the category
whose objects are $G$-graded $\ast$-algebras and its morphisms are given by
\[
    kk_{G_{\gr}}^h(A,B) := kk^h_{\GgrAlg}(M_G A, M_G B).
\]
There is an analogous functor 
\[
    j^h_{G_{\gr}} \colon \GgrAlg \to kk_{G_{\gr}}^h
\] 
in the graded setting, which is the identity on objects and 
maps a $G$-graded morphism
$f \colon A \to B$ to the image of
$M_G f : M_G A \to M_G B$ via $j^h_{\GgrAlg}$, 
viewed as an arrow of $kk^h_{G_{\gr}}$.
\end{defn}

With these definitions in place,
the following theorems can be adapted 
from \cite{kkg}*{Theorem 4.1.1, Theorem 4.2.1} by means of Proposition 
\ref{prop:stab} and the universal property of $j_{\fA}^h$.

\begin{thm}[cf. {\cite{kkg}*{Theorem 4.1.1}}] 
\label{thm:kkgeq} The functor $j^h_G \colon \GAlg \to kk^h_G$
is initial in the cate--gory of $G$-stable,
homotopy invariant, $M_\fX$-stable and $\iota_+$-stable excisive 
homology theories.
\qed
\end{thm}

\begin{thm}[cf. {\cite{kkg}*{Theorem 4.2.1}}] 
\label{thm:kkggr} The functor $j^h_{G_{\gr}} \colon 
\GgrAlg \to kk_{G_{\gr}}^h$ is initial in the category 
of $G$-stable, homotopy invariant, $M_\fX$-stable and $\iota_+$-stable excisive 
homology teories.
\qed
\end{thm}

\section{Adjointness theorems}
\label{sec:adj}

In this section we show that the adjointness theorems of 
\cite{kkg} that relate algebraic 
bivariant $K$-theory with its equivariant and homogenous counterparts can be 
extended to the hermitian context. 

Composing the crossed product functors \eqref{fun:crossed} with the canonical functors to each 
of the bivariant $K$-theory categories and using their universal properties we obtain functors
\[
    - \rtimes G \colon kk^h_G \longleftrightarrow kk_{G_{\gr}}^h \colon G \hltimes -.
\]
By Proposition \ref{prop:impri} the composite functors $G \hltimes(- \rtimes G)$ and $(G\hltimes-)\rtimes G $
are naturally isomorphic to the endofunctors $M_G(-) \colon \GAlg \to \GAlg$ and
$M_G(-) \colon \GgrAlg \to \GgrAlg$ respectively. Consequently, 
the algebraic analogue of Baaj-Skandalis' duality proved in \cite{kkg} extends to
the hermitian setting as follows.

\begin{thm}[{cf. \cite{kkg}*{Theorem 7.6}}] \label{thm:baaj-skandalis}
The functors $- \rtimes G$ and $G \hltimes -$ extend to inverse triangle equivalences
\begin{center}
    \begin{tikzcd}
       kk^h_G \arrow[bend left=15]{rr}{- \rtimes G}
      &
      & kk_{G_{\gr}}^h \arrow[bend left=15]{ll}{G \hltimes -}
    \end{tikzcd}.
\end{center}
\qed
\end{thm}

The next adjointness result concerns induction and restriction. Given a subgroup $H \leq G$
and a $G$-$\ast$-algebra $A$ we will write $\res^G_H(A)$ for the $H$-$\ast$-algebra 
resulting from restricting the $G$-action on $A$ to an $H$-action. Write 
$\pi \colon G \to G/H$ for the canonical quotient function, and set
\[
    \ind^G_H(A) = \{f \in A^{(G)} : |\pi(\supp(f))| < \infty, \ f(g) = h \cdot f(gh) 
    \ (\forall g \in G, h \in H) \}.
\]
This is a $G$-$\ast$-algebra with product given by pointwise multiplication, 
and involution and action defined by $f^\ast(s) := f(s)^\ast$ and 
$g \cdot f(s) = f(g^{-1}s)$ for each $f \in \ind^G_H(A)$ 
and $g,s \in G$ respectively. 
We also define $\ind_H^G(\varphi)(f) = \varphi\circ f$ for each $H$-$\ast$-algebra 
morphism $\varphi \colon A \to B$.

\begin{ex} If $A$ is a $G$-$\ast$-algebra such that 
the restricted $H$-action is trivial, then
\[
    \ind^G_H\res^G_H(A) \cong A^{(G/H)}.    
\]  
Notice, in particular, that 
\begin{equation}\label{eq:ind=cross}
\ind_1^G\res_{1}^G(\ell)\cong \ell^{(G)}=G\hltimes\ell
\end{equation}
\end{ex}

Both $\res^G_H$ and $\ind^G_H$ extend to triangulated functors
\[
\res^G_H \colon kk^h_G \longleftrightarrow kk^h_H \colon \ind^G_H.    
\]
The following theorem can be adapted from \cite{kkg} 
by verifying that the algebra maps involved
preserve involutions.

\begin{thm}[{cf. \cite{kkg}*{Theorem 6.1.4}}] \label{thm:ind-res}
For each subgroup $H \leq G$ there exists an adjunction
\begin{center}
    \begin{tikzcd}
       kk^h_H \arrow[bend left=15]{rr}{\ind_H^G}
      & \rotatebox{90}{$\vdash$}
      & kk^h_G \arrow[bend left=15]{ll}{\res^G_H}
    \end{tikzcd}.
\end{center}
\qed
\end{thm}

\begin{rmk} \label{rmk:res-unit} The proof of \cite{kkg}*{Theorem 6.1.4} makes clear that 
when applying Theorem \ref{thm:ind-res} with $H = \{1\}$, the map 
$kk_G^h(A^{(G)}, B) \overset{\cong}{\lra} kk^h(A,B)$ is induced by 
forgetting the $G$-action and precomposing with the inclusion
$\iota_1 \colon a \in A \mapsto a \cdot \chi_1 \in  A^{(G)}$.
\end{rmk}
In the next corollary and elsewhere we write $KH_*^h$ for homotopy hermitian $K$-theory, as defined in \cite{kkh}*{Section 3}.

\begin{coro} \label{coro:kkhgr-ell=khcross} For each $G$-graded $\ast$-algebra $A$, 
there is an isomorphism of abelian groups 
\[
kk_{G_{\gr}}^h(\ell, A) \cong KH_0^{h,G_{\gr}}(A).
\]
\end{coro}
\begin{proof}
We have the following sequence of natural isomorphisms, where we have used Theorem \ref{thm:baaj-skandalis} at the first step, \eqref{eq:ind=cross} and Theorem \ref{thm:ind-res} at the second, the natural isomorphism $kk^h(\ell,-)\cong KH_0^h$ of \cite{kkh}*{Proposition 8.1} at the third, and \eqref{eq:khgr} at the last.
\[
kk_{G_{\gr}}^h(\ell, A)\cong kk_G^h(\ell^{(G)}, G\hltimes A)\cong kk^h(\ell,G\hltimes A)\cong KH^h_0(G\hltimes A)=KH_0^{h,G_{\gr}}(A).
\]
\end{proof} 
The groups $KH^{h,G_{\gr}}_n(A)$ for $n\ne 0$ 
can be interpreted as $kk^h_{G_{\gr}}(\ell,-)$
applied to $\Omega^nA$ for $n>0$ and to the iterated Karoubi suspension 
$\Sigma^{-n}A$ if $n<0$, see \cite{kk}*{Subsection 4.7} and \cite{kk}*{Corollary 6.4.2}.

\begin{rmk}
\label{rmk:adj} Both the 
algebraic analogue of the Green-Julg theorem 
proved in \cite{kkg} and the bivariant version 
of Karoubi's $12$-term exact sequence given in \cite{kkh}
can be adapted to the present context. As in the theorems 
above, the proofs consist in verifying that all 
algebras and maps involved preserve the extra structure 
considered.
\end{rmk}
\section{The coefficient ring \topdf{$kk^h_{G_{\gr}}(\ell,\ell)$}{kkhGgr(l,l)} and \topdf{$G$}{G}-action on graded \topdf{$K$}{K}-theory}
\label{sec:enrich}
We have the following chain of isomorphisms of abelian groups, where we use agreement between $KH^h_0$ and $kk^h(\ell,-)$ \cite{kkh}*{Proposition 8.1} for the first and third isomorphisms, additivity of $KH_0^h$ for the second,  Theorem \ref{thm:baaj-skandalis} for the fourth and identity \eqref{eq:ind=cross} and Theorem \ref{thm:ind-res} for the last 
\begin{gather}\label{map:isopp}
\Z[G]\otimes kk^h(\ell,\ell)\cong \Z[G]\otimes KH_0^h(\ell)=\bigoplus_{g\in G}\Z g\otimes KH_0^h(\ell)\\ \cong KH_0^h(\ell^{(G)})
\cong kk^h(\ell,\ell^{(G)})\cong kk^{h}_{G}(\ell^{(G)},\ell^{(G)})\cong kk^h_{G_{\gr}}(\ell,\ell).\nonumber
\end{gather}
Observe that both the domain and codomain of the composite isomorphism are rings; the next proposition says that the map is an anti-homomorphism. 
\begin{prop} 
\label{prop:end-el}    
The composite \eqref{map:isopp} is a ring isomorphism
\begin{equation}\label{eq:iso-kkhgr-ell}
    \Z[G^{\op}] \otimes  kk^h(\ell,\ell) \overset{\sim}{\lra}kk^h_{G_{\gr}}(\ell, \ell),
\end{equation}
and maps
\begin{equation}\label{map:elg}
g\otimes j^h_{G_{\gr}}(\id_\ell)\mapsto m_g:=j^h_{G_{\gr}}(\iota_1)^{-1}\circ j^h_{G_{\gr}}(\iota_g).
\end{equation}
\end{prop}
\begin{proof} 
The last isomorphism in the composite \eqref{map:isopp} comes from a linear functor, so it is a ring homomorphism. Hence it suffices to show that the composite of all maps but the last \begin{equation}\label{eq:invenrich}
\Z[G^{\op}] \otimes_\Z kk^h(\ell,\ell)\iso kk^h_G(\ell^{(G)},\ell^{(G)})
\end{equation}
is a ring homomorphism. This is immediate upon checking that \eqref{eq:invenrich} maps an elementary tensor $g \otimes [1]$ to the
$kk^h$ class of 
\begin{equation}\label{map:rogel}
\rho_g:\ell^{(G)}\to\ell^{(G)},\,\, \rho_g(\chi_s)=\chi_{sg},
 \end{equation}
and an elementary 
tensor $1 \otimes \xi$ with $\xi \in KH^h_0(\ell)=kk^h(\ell,\ell)$ to $\xi^{(G)} = \ind_{1}^G(\xi)$. 
Next, one checks, using the explicit formula of  \cite{kkg}*{Proposition 7.4}, that under the isomorphism
$G\hltimes\ell^{(G)}\cong M_G$ of Proposition \ref{prop:impri}, $G\hltimes\rho_g$ identifies with the graded isomorphism $R_g:M_G\to M_G$, $R_g(\elmat_{s,t})=\elmat_{sg,tg}$. Thus the isomorphism $kk_G^h(\ell^{(G)},\ell^{(G)}) \cong kk^h_{G_{\gr}}(\ell, \ell)$ sends 
the class of $\rho_g$  to $j^h_{G_{\gr}}(\iota_1)^{-1}j^h_{G_{\gr}}(R_g\iota_1)= j^h_{G_{\gr}}(\iota_1)^{-1}j^h_{G_{\gr}}(\iota_g)$. 
\end{proof}

It follows from Proposition \ref{prop:end-el} that we have a left action 
\[
\cdot:G\times kk^{h}_{G_{\gr}}(\ell,A)\to kk^h_{G_{\gr}}(\ell,A),\,\, g\cdot \xi=\xi\circ m_g.
\]
In addition, $G$ acts on $G\hltimes A$ by automorphisms; this gives a second action 
\[
\cdot':G\times kk^{h}(\ell,G\hltimes A)\to kk^h(\ell,G\hltimes A).
\]
Using Theorems \ref{thm:ind-res} and \ref{thm:baaj-skandalis} we obtain an isomorphism 
\begin{equation}\label{map:elmu}
\mu:kk^{h}_{G_{\gr}}(\ell,A)\iso kk^{h}(\ell,G\hltimes A)
\end{equation}
\begin{lem}\label{lem:cdot=cdot'}
The actions $\cdot$ and $\cdot'$ correspond to each other under the isomorphism \eqref{map:elmu}: we have $\mu(g\cdot \xi)=g\cdot'\mu(\xi)$ for all $g\in G$, $\xi\in kk^h_{G_{\gr}}(\ell,A)$.
\end{lem}
\begin{proof}
By the proof of Proposition \ref{prop:end-el}, $G\hltimes m_g$ is the class of
 the endomorphism $\rho_g$ of \eqref{map:rogel}. Hence it suffices to show that for $\xi\in kk^h_{G_{\gr}}(\ell,A)$, we have 
 \begin{equation}\label{eq:cdot}
 g\cdot'(G\hltimes \xi)=(G\hltimes \xi)\circ\rho_g.
 \end{equation}
 It is clear that \eqref{eq:cdot} holds whenever $\xi$ is the class of a homogeneous projection $p\in M_GM_\infty A$. From this one obtains that \eqref{eq:cdot} also holds when $\xi$ comes from a quasi-homomorphism $\ell\rightrightarrows M_GM_\infty \tilde{A}\vartriangleright M_GM_\infty \tilde{A}$ in $G_{\gr}-\ahas$. The general case follows from the latter case using Corollary \ref{coro:kkhgr-ell=khcross} together with \eqref{eq:khgr}, Definition \ref{def:khgr} and naturality. 
\end{proof}

By the same argument as in \cite{kkh}*{Lemma 6.2.13}, the tensor
 product of $G$-$\ast$-algebras extends to a 
product on the bivariant $K$-theory category $kk_{G}^h$. The graded case follows similarly when the group is abelian (see Remark \ref{rmk:tensor-gr}) and we have the following.

\begin{lem}\label{lem:tensograd} Let $G$ be an abelian group and let  $A_1, A_2, B_1,B_2\in G_{\gr}-\ahas$. There is a natural, associative, bilinear product 
\begin{gather*}
\otimes \colon kk_{G_{\gr}}^h(A_1, A_2) \times kk_{G_{\gr}}^h(B_1,B_2) \lra
kk_{G_{\gr}}^h(A_1 \otimes_\ell B_1, A_2 \otimes_\ell B_2),\\
\xi\otimes\eta:=(\xi\otimes B_2)\circ (A_1\otimes \eta).
\end{gather*}
which is compatible with composition and satisfies 
$j^h_{G_{\gr}}(f \otimes g) = j^h_{G_{\gr}}(f) \otimes j^h_{G_{\gr}}(g)$.
\qed
\end{lem}

\begin{prop}\label{prop:multiabel} Let $G$ be an abelian group, $g\in G$, $A,B\in G_{\gr}-\ahas$ and $\xi \in kk_{G_{\gr}}^h(A ,B)$. Omitting $j^h_{G_{\gr}}$, we have a commutative diagram in $kk_{G_{\gr}}^h$
$$\xymatrix{A\ar[r]^{\xi \otimes g}\ar[d]^{\iota_g}&B\\
M_G A \ar[r]^{\iota_1^{-1}} &A \ar[u]^{\xi} }$$
\end{prop}
\begin{proof}
Immediate from Proposition \ref{prop:end-el} and Lemma \ref{lem:tensograd}.
\end{proof}
\section{Bivariant Dade theorem}\label{sec:dade}
In this section we show that for a strongly graded unital $*$-algebra $R$, 
the map $R_1\to  G \hltimes R$, $r\mapsto \chi_1 \hltimes r$, is a 
$kk^h$-equivalence. In view of Theorem \ref{thm:kh-cross=kh-gr} and Remark 
\ref{rmk:dade-h-cross}, we regard this result as a bivariant version of the 
implication i)$\Rightarrow$ii) of hermitian Dade Theorem \ref{thm:h-dade}. 

\begin{thm}\label{thm:dade} Let $R\in G_{\gr}-\ahas$ be a strongly graded unital $*$-algebra. Then the map $R_1\to G\hltimes R$, $a\mapsto \chi_1\hltimes a$
is a $kk^h$-equivalence. 
\end{thm}
The proof of Theorem \ref{thm:dade} will be given at the end of the section. 

\begin{coro}\label{coro:dade}
Let $A\in G_{\gr}-\ahas$; assume either that $A$ is trivially graded or that $G$ is abelian. Then the map
$A\to G\hltimes A[G]$, $a\mapsto \chi_1\hltimes (\sum_{g\in G}a_gg^{-1})$ is a $kk^h$-equivalence.
\end{coro}
\begin{proof}
If $A$ is unital, then by Example \ref{ex:groupalg}, $A[G]$ is strongly graded and $A\to A[G]_1$, $a\mapsto \sum_{g\in G}a_gg^{-1}$ is a $*$-isomorphism, whence the corollary follows from Theorem \ref{thm:dade}. The corollary for general $A$ follows from the unital case applied to $\ell$ and to the $\ell$-algebra unitalization $\tilde{A}$, using excision. 
\end{proof}
\begin{lem}\label{lem:moritedade}
Let $S\in\ahas$ and $\cE\subset S$ a set of nonzero orthogonal projections of cardinality $\#\cE\le \#\fX$. Assume that 
$$\cU=\{\sum_{e\in \cF}e: \cE\supset \cF\text{ finite }\}$$
is a set of local units for $S$. Further assume that $\cE$ contains a full projection $e_1$. Then the inclusion $e_1Se_1\to S$ is a $kk^h$-equivalence. 
\end{lem}
\begin{proof}
Because the elements of $\cE$ are orthogonal projections and $\cU$ is a set of local units, we have $S=\bigoplus_{e,f\in \cE}eSf$. Thus for $C_{\cE}$ as in \eqref{eq:wacon} and $p=\sum_{e\in\cE}\elmat_{e,e}\otimes e\in C_{\cE}S$, we have $S\cong pM_{\cE}Sp$. 
By hypothesis, for each $e\in \cE$ there are a finite row matrix $r_e\in (eSe_1)^{(\{1\}\times \N)}$ and a finite column matrix $c_e\in (e_1Se)^{(\N\times \{1\})}$  such that $r_ec_e=e$; in particular we may choose $r_{e_1}=c_{e_1}=\elmat_{1,1}e_1$. Let $r,c\in C_{\N\times\cE}S$, $r=\sum_{n,e}\elmat_{(1,e),(n,e)}(r_e)_{1,n}$ and $c=\sum_{n,e}\elmat_{(n,e),(1,e)}(c_e)_{n,1}$. We have  $rc=\sum_{e\in\cE}\elmat_{(1,e),(1,e)}e$, which is the image of $p$ under the inclusion $\alpha:C_{\cE}S\subset C_{\N\times \cE}S$ induced by $\cE\to \N\times\cE$, $e\mapsto (1,e)$. Let $\lambda$ be as in \eqref{eq:lambda} and set
\[
u=\begin{bmatrix}\lambda^*c+\lambda r^*& c-r^*\\
  \lambda\lambda^*(c-r^*)&\lambda c+\lambda^*r^*\\
  \end{bmatrix}\in M_{\pm}C_{\N\times \cE}S.
\]
For $A\in\ahas$, let $\iota_-:A\to M_{\pm}A$ be the lower right corner inclusion. One checks that $u^*u=\iota_+(\alpha(p))+\iota_{-}(\alpha(p))$. Set 
\[
R=e_1Se_1,\,\,T=\alpha(p) M_\infty M_{\cE} S\alpha(p).
\]
The composite
\[
\phi:S\cong T\overset{\iota_+}{\lra} M_{\pm}T\overset{\ad(u)}{\lra}M_{\pm}M_\infty M_{\cE}R
\]
is a $\ast$-algebra homomorphism. One checks that the composite of $\phi$ and the inclusion $\inc:R\subset S$ is the corner inclusion $\iota_+\otimes\iota_1\otimes\iota_{e_1}$ and is thus a $kk^h$ equivalence by matricial and hermitian stability. The composite
$M_\pm M_\infty M_{\cE}(\inc)\circ\phi$ maps $a\mapsto \ad(u)(\iota_+(a))$, and is thus a $kk^h$-equivalence by \cite{kkhlpa}*{Lemma 8.12} and hermitian stability of the functor $j^h$. This finishes the proof. 
\end{proof}

\begin{proof}[Proof of Theorem \ref{thm:dade}] Set $S:=G\hltimes R$, $\cE=\{\chi_g\hltimes 1: g\in G\}$ and $e_g=\chi_g\hltimes 1$. Because $R$ is strongly graded, for every $g\in G$ there exist $n\ge 1$, $y\in R_{g^{-1}}^{1\times n}$ and $y'\in R_{g}^{n\times 1}$ such that $yy'=1$. Then $x=\chi_g\hltimes y:=(\chi_g\hltimes y_1,\dots,\chi_g\hltimes y_n)\in (\chi_g\hltimes R_{g^{-1}})^{1\times n}=e_gS^{1\times n}e_1$ and $x'=\chi_1 \hltimes y'$ satisfy $e_g=xx'$. Now apply Lemma \ref{lem:moritedade}.
\end{proof}

\begin{rem}\label{rem:dade}
We do not know whether the converse of Theorem \ref{thm:dade} holds. However it is straightforward to show that the key property that we use in the proof, namely, that the projection $\chi_1\hltimes 1\in G\hltimes R$ is full, is equivalent to the grading of $R$ being strong. \end{rem}

\section{A distinguished triangle for Leavitt path algebras in \topdf{$kk_{G_{\gr}}^h$}{kkhGgr}}
\label{sec:triang}

Let $E$ be a graph; write $C(E)$ and for its Cohn algebra. For $v \in E^0$ such that $|s^{-1}(v)| < \infty$, put
\begin{equation}\label{eq:qvmv}
    m_v := \sum_{e \in s^{-1}(v)}ee^\ast, \quad q_v = v-m_v \in C(E).
\end{equation}
Recall that by definition we have a short 
exact sequence \eqref{ext:cohn}
where $\cK(E) = \langle q_v \colon v \in \reg(E) \rangle$.

Let $E$ be a graph and 
$\omega \colon E^1 \to G$ a weight; it induces a 
$G$-grading on $C(E)$ by defining $|e| = \omega(e)$,
$|e^\ast| = \omega(e)^{-1}$ and
$|v| = 1$ for each $v \in E^0, e \in E^1$. Moreover, 
the ideal $\cK (E)$ is homogeneous
with respect to this grading, as it is generated by homogenous
elements of degree $1$. In particular, 
the Leavitt path algebra $L(E)$ inherits a $G$-grading
associated to $\omega$. We write $\cK_\omega(E)$, $C_\omega (E)$ and $L_\omega(E)$ for the algebras $\cK(E)$, $C(E)$ and $L(E)$ equipped with these $G$-gradings. It follows from 
\cite{lpabook}*{Proposition 1.5.11} that
\eqref{ext:cohn} is $\ell$-split and, moreover, that 
the section preserves the $G$-gradings induced by $\omega$.

Fix a graph $E$; a homology theory $H:G_{\gr}-\ahas\to \cT$ is called \emph{$E$-stable} if it is $\iota_+$-stable and $G$-stable in the sense of Convention \ref{conv:gstab} with respect to $G$-graded sets of cardinality no higher than that of $\fX=E^0\sqcup E^1\sqcup \N\sqcup G$. 

If $I$ is a set, $\cA$ an additive category and $F:G_{\gr}-\ahas\to\cA$ a functor, we say that $F$ is \emph{$I$-additive} if first of all direct sums of cardinality $\le \#I$ exist in $\cA$ and second of all the map
\[
\bigoplus_{j\in J}F(A_j)\to F(\bigoplus_{j\in J}A_j)
\]
is an isomorphism for any family $\{A_j:j\in J\}\subset G_{\gr}-\ahas$ with $\#J\le\# I$. 

\begin{prop} \label{prop:q-iso} Let $E$ be a 
graph and $\omega \colon E^1 \to G$ a weight. 
Consider the graded $\ast$-homomorphism
\begin{equation} \label{mor:q}
    q \colon \ell^{\reg(E)} \to \cK_\omega (E), \quad \chi_v \mapsto q_v.
\end{equation}
Let $\cA$ be an additive category and $F:G_{\gr}-\ahas\to\cA$ a $\reg(E)$-additive, $E$-stable functor. Then $F(q)$ is an isomorphism.
\end{prop}
\begin{proof} For each $v \in E^0$, write $\cP_v$ for the set of paths that end at $v$. We may view any path in $E$ as homogeneous element of $C_\omega(E)$; this equips $\cP_v$ with a $G$-grading, and the canonical $\ast$-algebra isomorphism 
\begin{equation}\label{map:matriso}
    \bigoplus_{v \in \reg(E)} M_{\cP_v} \iso \cK_\omega(E), \quad  
    \elmat_{\alpha, \beta} \mapsto \alpha q_v \beta^\ast, \quad (v \in E^0, 
    \alpha, \beta \in \cP_v)
\end{equation}
is $G$-graded. Since $q$ factors through \eqref{map:matriso} via the sum of inclusions 
$\oplus \iota_v \colon \ell^{\reg(E)} \to  \bigoplus_{v \in \reg(E)} M_{\cP_v}$ and since $F$ is $\reg(E)$-additive by hypothesis,
it suffices to show that $F(\iota_v)$ is an isomorphism for every $v\in\reg(E)$.
This follows from $E$-stability, as the morphism in question is induced by the 
$G$-graded set inclusion $\{v\} \subset \cP_v$ and $\#\cP_v\le \#(E^0\sqcup E^1\sqcup\N)$.
\end{proof}

\begin{prop} \label{prop:phi-iso} Let $E$ be a graph and $\omega \colon E^1 \to G$ a weight. Consider the $G$-graded $\ast$-homomorphism
\begin{equation}\label{mor:phi}
    \varphi \colon \ell^{(E^0)} \to C_\omega (E), \quad \chi_v \mapsto v.
\end{equation}
Let $H:G_{\gr}-\ahas\to\cT$ be a homotopy invariant, excisive, $E$-stable and $E^0$-additive homology theory. Then $H(\varphi)$ is an isomorphism. 
\end{prop}
\begin{proof} The proof of this proposition consists of verifying that the argument of  \cite{kklpa1}*{Theorem 4.2} for the
ungraded case, together with the necessary adaptations to the hermitian context \cite{kkhlpa}*{Theorem 3.2} follow 
through, once we equip all algebras and morphisms with their corresponding $G$-gradings induced by $\omega$. The ungraded, non-hermitian case is based upon the usage of quasimorphisms, which can be 
adapted to the present context from \cite{cmrkk}*{Proposition 3.3} in the same way. 
The homotopies given in \cite{kklpa1}*{Theorem 4.2} do not preserve involutions, 
hence they need to be adapted according to \cite{kkhlpa}*{Theorem 3.2}. One checks 
that moreover these elementary polynomial homotopies are $G$-graded. To conclude, we record the $G$-grading of the algebras involved in the proof, 
for which all maps of \cite{kklpa1}*{Section 4} are $G$-graded. 
Write $\cP$ for the set of paths of the graph $E$, equipped by the $G$-grading induced by $\omega$. In \cite{kklpa1}*{proof of Theorem 4.2, part I} an algebra $C^m(E)$ is introduced,
together with a injective $\ast$-homomorphism $\rho \colon C^m(E) \to \Gamma_\cP$.
One checks that the image of $\rho$ lies in $C^\circ_\cP$ and equips $C^m(E)$
with the induced $G$-grading. The
subalgebra $\fA$ defined in \cite{kklpa1}*{proof of Theorem 4.2, part II}
is a homogeneous ideal of $M_\cP C_\omega(E)$ and thus
carries a canonical $G$-grading. In the same proof a crossed 
product $\fA \ltimes_{\rho'} C^m(E)$ is considered, 
together with an ideal $J \vartriangleleft\fA \ltimes_{\rho'} C^m(E)$. 
As an $\ell$-module the former coincides with 
$\fA \oplus C^m(E)$; the $G$-grading given by $(\fA \ltimes_{\rho'} C^m(E))_g 
= \fA_g \oplus C^m(E)_g$ makes it into a $G$-graded $\ast$-algebra for which
the ideal $J$ is homogeneous.
\end{proof}

\begin{lem} \label{lem:iso-v-vg} 
Let $E$ be a graph, $\omega:E^1\to G$ a weight, and $\widetilde{E}=\widetilde{(E,\omega)}$. For each $v \in E^0$ and $g \in G$, the isomorphism
\[
    kk_{G_{\gr}}^h(\ell, C_\omega(E)) \cong kk_G^h(\ell^{(G)}, G \hltimes C_\omega(E))
    = kk_G^h(\ell^{(G)}, C(\ucov{E})) \cong kk^h(\ell, C(\ucov{E}))
\] 
maps $g \cdot j^h_{G_{\gr}}(v)$ to $j^h(v_g)$.
\end{lem}
\begin{proof} Apply Lemma \ref{lem:cdot=cdot'} to the projection $v \in C_\omega(E)$.
\end{proof}

\begin{rem}\label{rem:GactsH}
Let $H:G_{\gr}-\ahas\to\cT$ be a homotopy invariant, $E$-stable, excisive homology theory. Then by the universal property of $j^h_{G_{\gr}}$ (Theorem \ref{thm:kkggr}),  there is a ring homomorphism 
$\bar{H}\colon kk^h_{G_{\gr}}(\ell,\ell)\to \End_{\cT}(\ell)$. Composing it with the map of Proposition \ref{prop:end-el}, we obtain a ring homomorphism
\[
\Z[G^{\op}]\to \End_{\cT}(H(\ell)).
\]
In particular, if $H$ is $E^0$-additive, then $I-A_\omega^t$ defines a homomorphism $H(\ell)^{\reg(E)}\to H(\ell)^{E^0}$ in $\cT$.
\end{rem}

\begin{thm}\label{thm:triang}
Let $E$ be a graph, $G$ a group, $\omega:E^1\to G$ a weight, and $A_\omega\in\Z[G]^{(\reg(E)\times E^0)}$ the weighted adjacency matrix of \eqref{intro:weightind}. Let $H:G_{\gr}-\ahas\to\cT$ be a homotopy invariant, $E$-stable, excisive and $E^0$-additive homology theory. Then $H$ maps the Cohn extension to a distinguished
triangle in $\cT$ of the form
\begin{equation}\label{triang:general}
H(\ell)^{\reg(E)} \xto{I-A_\omega^t} H(\ell)^{E^0} \to H(L_\omega(E)).
\end{equation}
\end{thm}
\begin{proof} Let $\inc:\cK_\omega(E)\to C_\omega(E)$ be the inclusion map. By Propositions \ref{prop:phi-iso} and \ref{prop:q-iso} together with extension
\eqref{ext:cohn}, for $\xi = H(\varphi)^{-1}\circ H(\inc \circ q)$ we have 
 a distinguished triangle 
\[
H(\ell)^{(\reg(E))} \xto{\xi} H(\ell)^{(E^0)} \to H(L_\omega(E)).
\]
The proof consists in proving that $\xi$ is multiplication by $I-A_\omega^t$ by equivalently showing that
$H(\inc\circ q) = H(\varphi)\circ (I-A_\omega^t)$. By the additivity hypothesis on $H$, it suffices to prove that for each inclusion $\iota_v \colon 
\ell \mapsto \chi_v \ell \subset \ell^{\reg(E)}$ we have 
that $H(\inc q \iota_v) = H(\varphi) (I-A_\omega^t) H(\iota_v)$. By the universal property of $j^h_{G_{\gr}}$, we are reduced to proving that
\[
j^h_{G_{\gr}}(q_v) = j^h_{G_{\gr}}(v) - 
\sum_{e \in s^{-1}(v)} \omega(e) \cdot j^h_{G_{\gr}}(r(e)).
\]
By Lemma \ref{lem:iso-v-vg}, this is equivalent to having
\begin{equation} \label{eq:cov-v-kkh}
j^h(q_{v_1}) = j^h(v_1) - 
\sum_{e \in s^{-1}(v_1)} j^h(r(e)_{\omega(e)})   
\end{equation}
in $kk^h(\ell, C(\ucov{E}))$, which follows from 
\cite{kkhlpa}*{proof of Theorem 3.4} applied to $v_1 \in \ucov{E}^0$.
\end{proof}

Let $E$ be a graph, $G$ a group, and $\omega:E^1\to G$ a weight. The graded \emph{Bowen-Franks} and \emph{dual Bowen-Franks} $G$-modules associated to $\omega$ are
\[
\gBF(E,\omega)=\coker(I-A_\omega^t),\,\, \gBF^\vee(E,\omega)=\coker(I^t-A_\omega).
\]
Observe that for $G=\Z$ and $c_n$ the constant weight $c_n(e)=n\in\Z$ for all $e\in E^1$, $\gBF(E,c_1)=\gBF(E)$ is the $\Z[\sigma]$-module of  \eqref{eq:bfgr}, while 
$\gBF(E,c_0)=\BF(E)$ is the usual Bowen-Franks group. 

\begin{coro}\label{coro:kle} Let $R\in\aha$; equip $R$ with the trivial $G$-grading and
$L_\omega(E)\otimes R$ with the tensor product grading. 

\item[i)] For each $n\in \Z$ there is a short exact sequence
\begin{equation}\label{seq:kle}
0\to\gBF(E,\omega)\otimes KH_{n}(R)\to  KH_n^{G_{\gr}}(L_\omega(E)\otimes R)\to \ker((I-A_\omega^t)\otimes KH_{n-1}(R))\to 0.
\end{equation}
For  $R\in\ahas$, the exact sequence \eqref{seq:kle} for $KH^{h,G_{\gr}}$ also holds.
\item[ii)] If $R$ is regular supercoherent and $G$ and $E$ are countable, then the canonical maps
$K_*(R)\to KH_*(R)$ and $K_*^{G_{\gr}}(L_\omega(E)\otimes R)\to KH_*^{G_{\gr}}(L_\omega(E)\otimes R)$ are isomorphisms. If in addition $R\in\ahas$ and $2$ is invertible in $\ell$, then we also have $K^h_*(R)\iso KH^h_*(R)$ and $K_*^{h,G_{\gr}}(L_\omega(E)\otimes R)\iso KH_*^{h,G_{\gr}}(L_\omega(E)\otimes R)$.
\end{coro}
\begin{proof}
The second assertion of part i) follows from Theorem \ref{thm:triang} applied to $KH^{h,G_{\gr}}(-\otimes R)$, using that, as $R$ is trivially graded,
$KH_*^{h,G_{\gr}}(R)=\Z[G]\otimes KH_*^h(R)$; the first is the particular case with $\inv(R)$ substituted for $R$. By \cite{kkhlpa}*{Lemma 4.3}, if $F$ is a countable graph and $R$ is unital and regular supercoherent then $L(F)\otimes R$ is $K$-regular, and even $K^h$-regular if $2$ is invertible in $R$. In particular this applies when $F$ is the covering graph of $E$ or the edgeless graph on one vertex. Combining this with part i) and using Theorems \ref{thm:k-cross=k-gr} and \ref{thm:kh-cross=kh-gr} and Proposition \ref{prop:crosscover}, we obtain the remaining assertions.
\end{proof}
\begin{rem}\label{rem:kle} If $G$ is abelian, then $-\otimes R:G_{\gr}-\aha\to G_{\gr}-\aha$ is defined for every $R\in G_{\gr}-\ahas$, and a similar argument as that of Corollary \ref{coro:kle} shows that there is an exact sequence
\begin{equation}\label{seq:kle2}
0\to\gBF(E,\omega)\otimes_{\Z[G]} KH^{G_{\gr}}_{n}(R)\to  KH_n^{G_{\gr}}(L_\omega(E)\otimes R)\to \ker((I-A_\omega^t)\otimes_{\Z[G]} KH^{G_{\gr}}_{n-1}(R))\to 0.
\end{equation}
The rest of Corollary \ref{coro:kle}, including the hermitian and regular supercoherent versions of \eqref{seq:kle2} also extend to nontrivially graded $R$ in the obvious way, provided that $G$ is abelian. 

We remark that in the particular case $G=\Z$ and $\omega(e)=1$ for all 
$e\in E^1$, then $I-A_\omega^t=I-\sigma A^t$ is injective by Lemma 
\ref{lem:mononotepi} and the rightmost nonzero term of \eqref{seq:kle2} is 
$\tor^{\Z[\sigma]}_1(\BF_{\gr}(E),KH_{n-1}^{\Z_{\gr}}(R))$.  
\end{rem}
\begin{rem}\label{rem:can}
It follows from Corollary \ref{coro:kle} that there is a canonical homomorphism of $\Z[G]$-modules
\begin{equation}\label{map:can}
\can:\gBF(E,\omega)\to KH_0^{h,G_{\gr}}(L_\omega(E)),\,\, \can(x)=x\otimes[1]
\end{equation}
The map $\can$ is an isomorphism whenever $KH^{h}_0(\ell)=\Z$ and $KH^{h}_{-1}(\ell)=0$. Such is the case, for example, when $\ell=\inv(\ell_0)$ and $\ell_0$ is a field or a PID, and we have $K_0^{G_{\gr}}(L_\omega(E))=\gBF(E,\omega)$.  
\end{rem}

\begin{coro}\label{coro:g-triang}
Assume that $E^0$ is finite. Then there is a distinguished triangle in $kk^h_{G_{\gr}}$
\[
\ell^{\reg(E)} \xto{I-\sigma \cdot A_\omega^t}  \ell^{E^0} \to L_\omega(E).
\]
\end{coro}
\begin{proof}
Apply Theorem \ref{thm:triang} to $H=j^h_{G_{\gr}}$. 
\end{proof}
\begin{rmk} 
Applying Corollary \ref{coro:g-triang} to the unique
weight with codomain the trivial group, we
recover the distinguished triangle for $L(E)$ in 
hermitian bivariant $K$-theory given in \cite{kkhlpa}*{Theorem 3.6}, and by \cite{kkh}*{Example 6.2.11}, also that in algebraic bivariant $K$-theory proved in \cite{kklpa1}*{Proposition 5.2}.
\end{rmk}

Let $E$ be a graph with finitely many vertices and let $\can:\gBF(E,\omega)\to KH_0^{h,G_{\gr}}(L(E))$ be as in \eqref{map:can}. For each $G$-graded $\ast$-algebra $A$ we have a map
\[
\ev \colon kk^h_{G_{\gr}}(L(E),A) \to 
\hom_{\Z[G]}(\gBF(E,\omega),KH^{h, G_{\gr}}_0(A)), 
\quad \xi \mapsto KH^{h, G_{\gr}}_0(\xi) \circ \can.        
\]

\begin{coro}\label{coro:uct} 
Let $E$ be a graph with finitely many vertices and $\omega:E^1\to G$ a weight. For each $G$-graded $\ast$-algebra $A$, there is an exact sequence
\[
0 \to KH^{h, G_{\gr}}_1(A) \otimes_{\Z[G]} \gBF^\vee(E,\omega) \to kk^h_{G_{\gr}}(L_\omega(E),A) 
\xto{\ev} \hom_{\Z[G]}(\gBF(E,\omega),KH^{h, G_{\gr}}_0(A)) \to 0.
\]
\end{coro}
\begin{proof} Apply the argument of \cite{kkhlpa}*{Theorem 12.2} to the triangle 
of Corollary \ref{coro:g-triang} and the free $\Z[G]$-module presentations
\begin{gather*}
\Z[G]^{\reg(E)}\overset{I-A_\omega^t}{\lra}\Z[G]^{E^0}\to\gBF(E,\omega)\to 0,\\
\Z[G]^{E^0}\overset{I^t-A_\omega}{\lra}\Z[G]^{\reg(E)}\to\gBF^\vee(E,\omega)
\to 0.
\end{gather*}
\end{proof}

\section{A Van den Bergh triangle in \topdf{$kk^h$}{kkh}}
\label{sec:vdb}

In \cite{vdb}*{Theorem on page 1563} Van den Bergh proves that the graded 
$K$-theory of a $\Z$-graded noetherian regular ring $R$
is related to its ungraded $K$-theory by means of an 
exact sequence
\begin{equation}\label{sel:orig-vdb}
\cdots \to K_{n+1}(R) \to  K^{\Z_{\gr}}_{n}(R) \xto{1-\sigma} K^{\Z_{\gr}}_n(R)
\xto{\forg} K_n(R) \to \cdots
\end{equation}
Here $\sigma$ is the map induced by degree shift and 
$\forg$ is the map induced by the forgetful functor 
$\cat{Proj}_{\Z_{\gr}}(R) \to \cat{Proj}(R)$.

\begin{thm}\label{thm:vdb}
Let $A\in\Z_{\gr}-\aha^*$. Then there is a distinguished triangle in $kk^h$
\[
\xymatrix{
\Z\hltimes A\ar[rr]^{1-\Z\hltimes\sigma}&& \Z\hltimes A\ar[r] & A
}
\]
\end{thm}
\begin{proof} Corollary \ref{coro:g-triang} applied to $G=\Z$, the graph $\cR_1$ consisting of a single loop $e$, and the weight $\omega(e)=1$, yields a distinguished triangle in $kk^h_{\Z_{\gr}}$
\[
\ell\xto{1-\sigma} \ell\to \ell[t,t^{-1}]
\]
Tensor this triangle with $A$, apply the functor $\Z\hltimes-:kk^h_{\Z_{\gr}}\to kk^h_{\Z}$ followed by the forgetful functor $kk^h_{\Z}\to kk^h$, and use Corollary \ref{coro:dade}. 
\end{proof}

\begin{coro}\label{coro:vdb} Let $R\in\Z_{\gr}-\aha$, $S\in\Z_{\gr}-\ahas$ and $n\in\Z$. Then there are long exact sequences
\begin{gather}
  \xymatrix{KH_{n+1}(R)\ar[r]&KH^{\Z_{\gr}}_{n}(R)\ar[r]^{1-\sigma} &KH^{\Z_{\gr}}_n(R)\ar[r]& KH_n(R)}\label{lec:khvdb}\\
\xymatrix{KH^h_{n+1}(S)\ar[r]&KH^{h,\Z_{\gr}}_{n}(S)\ar[r]^{1-\sigma} &KH^{h,\Z_{\gr}}_n(S)\ar[r]& KH^h_n(S)}\label{lec:khvdb-her}
  \end{gather}
If $R$ is regular supercoherent, we may substitute $K$ and $K^{\Z_{gr}}$ for 
    $KH$ and $KH^{\Z_{gr}}$ in \eqref{lec:khvdb}. If $S$ is regular supercoherent and $2$ is invertible in $S$, 
    we may substitute $K^h$ and $K^{h,\Z_{gr}}$ for 
    $KH^h$ and $KH^{h,\Z_{gr}}$ in \eqref{lec:khvdb-her}.     
\end{coro}
\begin{proof} Observe that \eqref{lec:khvdb} is the particular case $S=\inv(R)$ of \eqref{lec:khvdb-her}. We prove \eqref{lec:khvdb-her}. By excision, we can assume that $R$ is unital. 
Applying $kk^h(\ell,-)$ to the triangle of Theorem \ref{thm:vdb}
we obtain \eqref{lec:khvdb-her}. The assertions about the regular supercoherent case follow from Corollary \ref{coro:kle} and hermitian Dade's theorem \ref{thm:h-dade}. 
\end{proof}

\begin{rem}\label{rem:vdb}
The exact sequences of Corollary \ref{coro:vdb} are similar to that of Van den Bergh \eqref{sel:orig-vdb}, except perhaps in the map going from graded to ungraded $K$-theory; we shall presently see that the two maps are equivalent up to appropriate identifications. Let $R$ be unital ring and $\inc: R\subset R[t,t^{-1}]$ the canonical inclusion. The forgetful functor $\Z_{\gr}-\cat{mod}_R\to \cat{mod}_R$ is naturally equivalent to the composite of the scalar extension along $\inc$ with the functor $\Z_{\gr}-\cat{mod}_{R[t,t^{-1}]}\to \cat{mod}_R$, $M\mapsto M_0$. Under the equivalence  \eqref{mor:mod-cat-iso} the latter functor corresponds to $\cat{mod}_{\Z\hltimes R}\to \cat{mod}_R$, $M\mapsto M\cdot\chi_0$, which is left inverse to the scalar extension along $a\mapsto \chi_0\hltimes a$ (see Remark \ref{rmk:dade-h-cross}). Summing up, the map $\forg$ in \eqref{sel:orig-vdb} corresponds to that in Corollary \ref{coro:vdb}, which is defined as the composite of those induced by the inclusion $R\subset R[t,t^{-1}]$ and by the $kk$-inverse of the inclusion as the homogeneous component of degree zero coming from Corollary \ref{coro:dade}. \end{rem}

\section{A graded classification result}
\label{sec:classif}

We now proceed to classify Leavitt path algebras as objects in $kk^h_{\Z_{\gr}}$
in terms of the Bowen-Franks $\Z[\sigma]$-modules of their associated 
graphs.

\begin{thm}\label{lem:liftkkz}
Let $E$ and $F$ be graphs with finitely many vertices. If $\xi \colon \gBF(E)\to \gBF(F)$
is a $\Z[\sigma]$-module isomorphism, then there exists an isomorphism $\overline{\xi} \colon L(E) \to L(F)$
in $kk^h_{\Z_{\gr}}$ such that $\ev(\overline{\xi}) = \can \circ \xi$.
\end{thm}
\begin{proof} By Lemma \ref{lem:mononotepi}, we have length one free $\Z[\sigma]$-module resolutions
of $\gBF(E)$ and $\gBF(F)$ which guarantee that the 
argument of \cite{kkhlpa}*{Lemma 6.4} can be carried out
under the present hypotheses. 
\end{proof}

\begin{thm} \label{thm:classif}
Assume that $KH_{-1}(\ell) = 0$ and that the canonical morphism $\Z \to KH_0(\ell)$ 
is an isomorphism. Let $E$ and $F$ be graphs with finitely many vertices. Then the following 
are equivalent:
\begin{itemize}
    \item[(i)] The algebras $L(E)$ and $L(F)$ are $kk^h_{\Z_{\gr}}$-isomorphic.
    \item[(ii)] The algebras $L(E)$ and $L(F)$ are $kk_{\Z_{\gr}}$-isomorphic.
    \item[(iii)] The $\Z[\sigma]$-modules $\gBF(E)$ and $\gBF(F)$ are isomorphic.  
\end{itemize}
\end{thm}
\begin{proof} The fact that (i) implies (ii) amounts to applying the forgetful functor 
$kk^h_{\Z_{\gr}} \to kk_{\Z_{\gr}}$. On the other hand, the assumptions
on $\ell$ imply that $kk_{\Z_{\gr}}(\ell, \ell) \cong \Z[\sigma]$; this together with the case $G=\Z$, $\omega=c_1$ of 
 Corollary \ref{coro:g-triang} shows that $kk_{\Z_{\gr}}(\ell, L(E)) \cong \gBF(E)$ and likewise for 
the analogue statement replacing $E$ by $F$. It follows that (ii) implies (iii).
Finally, (iii) implies (i) by Theorem \ref{lem:liftkkz}.
\end{proof}

\end{document}